\documentclass[12pt, letterpaper]{amsart}

\usepackage{cleveref,tikz}
\usetikzlibrary{tikzmark}

\usepackage{cleveref}
\usepackage{amsfonts,mathrsfs}
\usepackage{amssymb}
\usepackage{amsmath}
\usepackage{amsthm}
\usepackage{mathtools}
\usepackage{latexsym}
\usepackage{color}
\usepackage{comment}
\usepackage{esint}
\usepackage{dsfont}
\usepackage{upgreek}

\newtheorem{thm}{Theorem}[section]
\newtheorem{lemma}[thm]{Lemma}
\newtheorem{lem}[thm]{Lemma}
\newtheorem{pro}[thm]{Proposition}
\newtheorem{cor}[thm]{Corollary}

\newtheorem*{thm*}{Theorem}

\theoremstyle{remark}
\newtheorem{remark}[thm]{Remark}

\theoremstyle{definition}
\newtheorem{defi}[thm]{Definition}

\numberwithin{equation}{section}

\newcommand{\R}{\mathbb{R}}
\newcommand{\ep}{\epsilon}

\DeclareMathOperator{\re}{Re}

\newcommand{\eps}{\varepsilon}

\newcommand{\be}{\begin{equation}}
\newcommand{\ee}{\end{equation}}
\newcommand{\bee}{\begin{equation*}}
\newcommand{\eee}{\end{equation*}}
\newcommand{\bea}{\begin{eqnarray}}
\newcommand{\eea}{\end{eqnarray}}

\newcommand{\HH}{\dim_{\mathcal{H}}}

\renewcommand{\a }{\alpha }
\renewcommand{\b }{\beta }

\newcommand{\e }{\varepsilon }

\newcommand{\intbar}{\mathop{\int\makebox(-13.5,0){\rule[4pt]{.7em}{0.3pt}}%
\kern-6pt}\nolimits}

\newcommand{\norm}[1]{\left\|#1\right\|}
\DeclareMathOperator{\trace}{tr}

\DeclareMathOperator{\dist}{dist}
\DeclareMathOperator{\rango}{Rg}

\newcommand{\mT}{\mathcal T}
\newcommand{\bs}{\begin{split}}
\newcommand{\esplit}{\end{split}}
\newcommand{\mbR}{\mathbb R}
\newcommand{\mC}{\mathcal C}
\newcommand{\back}{\backslash}
\newcommand{\lp}{\left(}
\newcommand{\rp}{\right)}
\newcommand{\PP}{\mathfrak p}

\title[Gluing for a fully-non-linear equation]{A gluing construction of singular solutions for a fully non-linear equation in conformal geometry}

\author[M.F. Espinal]{Maria Fernanda Espinal}
\address{M.F. Espinal,
	\hfill\break\indent
CMM Center for Mathematical Modeling
\hfill\break\indent
Av. Beauchef 851, Santiago,  Chile}
\email{mfespinal@cmm.uchile.cl}

\author[M.d.M. Gonz\'alez]{Mar\'ia del Mar Gonz\'alez}
\address{M.d.M. Gonz\'alez,
\hfill\break\indent
Universidad Aut\'onoma de Madrid
\hfill\break\indent
Departamento de Matem\'aticas, and ICMAT.  28049 Madrid, Spain}
\email{mariamar.gonzalezn@uam.es}

\begin{document}

\begin{abstract}
In this paper we study the $\sigma_2$--Yamabe equation, $n>4$, for solutions with a prescribed singular set $\Lambda$ given by a disjoint union of closed submanifolds whose dimension is positive and strictly less than $(n-\sqrt{n}-2)/2$. The $\sigma_2$--curvature in conformal geometry is defined as the second elementary symmetric polynomial of the eigenvalues of the Schouten tensor, which yields a fully non-linear PDE for the conformal factor.  We show that the classical gluing method, used by Mazzeo-Pacard (JDG 1996) for the scalar curvature problem, can be used  in the fully non-linear setting. This is a consequence of the conformal properties of the $\sigma_2$ equation, which imply that the linearized operator has good mapping properties in weighted spaces.
\end{abstract}

\maketitle

\begin{center}
\end{center}

\bigskip

\centerline{\bf AMS subject classification:  35J75, 53C21,
58J60,35J96}


\section{Introduction}

Let $(M, g)$ be a compact smooth $n$--dimensional Riemannian manifold without boundary and let $1 \leq k <\frac{n}{2}$. Taking advantage of this second assumption, we introduce the following formalism for a conformal change of metric
\begin{equation}\label{conformal-metric-introduction}
{g}_u   :=  u^{\frac{4k}{n-2k}}\, {g} ,
\end{equation}
where the conformal factor $u>0$ is a smooth  positive function. In this context $ g$ will be referred as the background metric.

Let $Ric_g$, $R_g$ be the Ricci tensor and scalar curvature of $g$, respectively. The Schouten tensor with respect to the metric $g$ is given by
\begin{equation*}
A_g=\tfrac{1}{n-2} \left(Ric_g-\tfrac{1}{2(n-1)}R_g g\right).
\end{equation*}
For the conformal metric \eqref{conformal-metric-introduction}, the Schouten tensor of $g_u$ is related to the one of $g$ by the conformal transformation law
\begin{equation}\label{transformation-Schouten}\begin{split}
A_{g_{u}}  =  A_{g} -
\tfrac{2k}{n-2k}  u^{-1}{D^2 u}  +
\tfrac{2kn}{(n-2k)^2} u^{-2}  {\nabla u \otimes \nabla u}  -
\tfrac{2k^2}{(n-2k)^2}  u^{-2}  {|\nabla u|_g^2}  g  ,
\end{split}\end{equation}
where $D^2$ and $|\cdot|_g$ are computed with respect to the background metric $g$.

We define the $\sigma_k$--\emph{curvature} as the $k$-th elementary symmetric function of the eigenvalues of the $(1,1)$-tensor $g^{-1} A_{g}$
\begin{equation*}
\sigma_k(g^{-1} A_{g})=\sum_{i_1<\ldots<i_k}\lambda_{i_1}\ldots\lambda_{i_k},
\end{equation*}
 and the \emph{positive} $\Gamma_k^+ $ \emph{cone} as set of metrics
\begin{equation*}
\Gamma_k^+=\{g \,:\, \sigma_1(g^{-1}A_{g}),\ldots,\sigma_k(g^{-1}A_g)>0\}.
\end{equation*}
Fixed a background metric, the $\sigma_k$--\emph{Yamabe problem} consists in finding a conformal metric in the positive  $\Gamma_k^+$ cone of constant $\sigma_k$--curvature. This is a fully non-linear equation for the conformal factor $u$,
\begin{equation}\label{initial-equation}
\sigma_k \big(\, g_u^{-1} A_{g_u} \big)  =
2^{-k} \, \hbox{${n \choose k}$},
\end{equation}
and has been well studied in the compact setting (we cite, for instance, \cite{cgy1,ll,stw,gw,gv} although  by no means this list is complete).\\

\textbf{Statement of the results.} In this paper we examine the $\sigma_2$--Yamabe equation, which we rephrase as a PDE problem. 
For technical reasons, once a background metric $g$ has been fixed, we denote
\begin{equation}
\label{tensor-B}
B_{g_u}   :=  \tfrac{n-4}{4} \,  u^{\frac{2n}{n-4}} \, g_u^{-1}  A_{g_u}, \quad B_g:=\tfrac{n-4}{4} g^{-1}  A_{g}
\end{equation}
so that, for a conformal change given by  \eqref{conformal-metric-introduction},
\begin{equation}\label{formula-B}
B_{g_u}=u^2B_{g}+g^{-1}\left[ -uD^2u+\tfrac{n}{n-2k}\nabla u\otimes \nabla u-\tfrac{k}{n-2k}|\nabla u|_g^2 \,g\right],
\end{equation}
where all derivatives and norms are taken with respect to the background metric $g$. We can also write the positive $\Gamma_2^+$ cone as
\begin{equation*}
\Gamma_k^2=\{u>0 \,:\, \sigma_1(B_{g_u}),\sigma_2(B_{g_u})>0\}.
\end{equation*}
In this notation, the $\sigma_2$--equation \eqref{initial-equation} may be expressed  as a PDE problem:
\begin{equation}
\label{eq}
  \mathcal{N} (u, g)   :=  \sigma_2  \left( B_{g_{u}}
  \right)  - c |u|^{q-1}u  = 0 ,
\end{equation}
where we have set
\begin{equation*}
q:=\frac{4n}{n-4},\quad c:=\binom{n}{4}\lp\frac{n-4}{8}\rp^2. 
\end{equation*}
The advantage of this formulation  is that  \eqref{eq}  is well defined even if $u$ is non-positive.

\begin{thm}\label{main-theorem}
Let $(M,g_M)$ be a $n$-dimensional, compact, smooth, Riemannian manifold of positive $\sigma_2$--curvature (and in the positive $\Gamma_2^+$ cone), that is non-degenerate. Let $\Lambda$ be a subset of $M$ which is a closed, connected, submanifold of dimension $p$. Assume that
\begin{equation}\label{restriction-p}
0<p<\PP_2,
\end{equation}
where $\PP_2$ is given in  \eqref{exact-formulas} below. Then there exists an infinite dimensional family of solutions of \eqref{eq} in  $M\setminus\Lambda$, belonging to the positive $\Gamma_2^+$ cone, that are singular exactly on $\Lambda$. In particular, such solutions satisfy
\begin{equation}\label{asymptotic-behavior}
u(x)\sim \frac{1}{\dist(x,\Lambda)^{\frac{n-4}{4}}} \quad\text{as}\quad x\to\Lambda. 
\end{equation}
\end{thm}

\begin{remark}
Note that, for $k=2$, the restriction of being in the $\Gamma_2^+$ cone is easily satisfied since
$$\sigma_2\leq \frac{n-1}{2n}\sigma_1^2,$$
 no matter what the sign of $\sigma_1$ is (see \cite{Viaclovsky:contact} for a proof of this classical inequality).
\end{remark}

 In the proof of the theorem we follow the classical gluing method of Mazzeo-Pacard \cite{Mazzeo-Pacard} for the scalar curvature ($k=1$), which is a semilinear PDE. Our main contribution here is to show that their scheme can be adapted to the fully non-linear equation \eqref{eq} thanks to the conformal properties of the problem. The main idea is to construct an approximate solution $\bar u_\ep$ by gluing the standard singularity to the background manifold, and then apply a perturbation argument as $\ep\to 0$.

We remark that, even though this method could work for any $\sigma_k$ with $2\leq k<\frac{n}{2}$, we have some computational difficulties that restrict our main theorem to $k=2$; however,  we conjecture that it is still true for other values of $k$.  

If the solution satisfies $u>0$, then it defines a conformal metric $g_u=u^{\frac{8}{n-4}}g_M$, which is complete in $M\setminus\Lambda$, and has positive constant $\sigma_2$--curvature. Unfortunately, we only know that $u$ is positive in a neigbourhood of $\Lambda$, as our method does not yield positivity in general. This fact would require finer asymptotics in the neck region (compare, for instance, to the paper by Catino-Mazzieri \cite{cat-maz}, where they use a Schwarzschild metric in the neck). 

In any case, the objective of this paper is to understand the linearized problem for \eqref{eq}, which is an edge operator in the sense of Mazzeo \cite{Mazzeo:edge}. In particular, our main contribution is to show that it has good mapping properties in weighted spaces as a consequence of  the conformal character of the $\sigma_2$ equation.

Even if our arguments are mostly based on  the  Euclidean metric, similarly to the original paper of Mazzeo-Pacard \cite{Mazzeo-Pacard}, once a proof is available for the flat setting, then one may allow general manifolds $M$. However, in  Theorem \ref{main-theorem} we need to add the hypothesis of non-degeneracy of the background manifold $M$. This means that  the linearization \eqref{linope} of the $\sigma_2$ equation on $M$ is invertible. Observe that in the semilinear case non-degeneracy is also needed but it follows as soon as $M$ has non-negative scalar curvature (see Theorem 3 in \cite{Mazzeo-Pacard}). Moreover, this is a natural hypothesis in gluing problems  (see, for instance, Silva-Santos \cite{Silva-Santos} or Catino-Mazzieri \cite{cat-maz} on related gluing constructions for the  $\sigma_k$ equation).

Note also that, working with a compact background manifold $M$ we avoid  complications at infinity. This is equivalent to considering  a bounded  domain $\Omega\setminus\Lambda$ with Dirichlet boundary conditions on $\partial\Omega$ as in the original setting of Mazzeo-Pacard \cite{Mazzeo-Pacard}. In addition,  
 since our proof is local, the same theorem holds if $\Lambda$ is a finite union of disjoint submanifolds with the specified restrictions. 

Nevertheless, there are some difficulties coming from the non-linear structure of the equation that do not arise in the semi-linear case. Indeed, both in the gluing scheme and in the calculation of the linearized operator, one needs to take into account the full structure of the Schouten tensor, not just the Laplacian of the conformal factor.

Finally, let us make some comments on related bibliography. In the paper  Mazzieri-Segatti \cite{Mazzieri-Segatti}  they construct constant $\sigma_k$ metrics  with isolated singularities by gluing Delaunay-type metrics. In the publication Silva-Santos \cite{Silva-Santos}, the authors use asymptotic matching to find solutions to the  $\sigma_2$--Yamabe problem with isolated singularities. Other relevant papers are Guan-Lin-Wang \cite{Guan-Lin-Wang} and Catino-Mazzieri \cite{cat-maz}, on connected sum constructions for the $\sigma_k$--curvature. More recent papers are Duncan-Wang \cite{Duncan-Wang}, which contains complementary existence results, and \cite{Duncan-Nguyen:Loewner-Nirenberg1,Duncan-Nguyen:Loewner-Nirenberg2,Duncan-Nguyen:Loewner-Nirenberg3} on the $\sigma_k$ version of the Loewner-Nirenberg problem. 

We also remark that this gluing method is very versatile and can be applied in many other settings, for instance, in non-local problems \cite{ACDFGW,Chan:thesis}.\\

\textbf{Restrictions on $p$.}\label{subsection:restrictions-p}
In order to motivate hypothesis \eqref{restriction-p} on the dimension of $\Lambda$ in the main theorem, let us present first the model example of singular metrics of $\sigma_k$--curvature.

More precisely, take the canonical complete metric in $\mathbb S^n\setminus \mathbb S^p$, which has  constant $\sigma_k$--curvature and is singular  on $\Lambda=\mathbb S^p$. It is  conformal to the product $\mathbb S^{n-p-1}\times\mathbb{H}^{p+1}$ with its standard metric (a picture can be found in  \cite[Figure 1]{Bettiol-Gonzalez-Maalaoui}). Its Schouten tensor is diagonal and, modulo a multiplicative factor of $1/2$,
its eigenvalues are $1$ and $-1$, with multiplicities $n-p-1$ and $p+1$, respectively. Then we can compute
\be\label{model-sigma-k}2^k\sigma_k(\mathbb{H}^{p+1}\times \mathbb S^{n-p-1})
=\sum_{i=0}^{k} \binom {n-p-1} {i}\binom{p+1}{k-i}(-1)^{k-i}=:c_{n,p,k}.\ee
For point singularities ($p=0$) there is a rich geometry of Delaunay-type singular metrics (Chang-Han-Yang \cite{Chang-Han-Yang},  Li \cite{Li}, Han-Li-Teixeira \cite{Han-Li-Teixeira}, Li-Han \cite{Li-Han}).

 In this paper we will restrict to higher dimensional singularities, that is, $p>0$. We set
\begin{equation*}\label{set-dimension} \PP_k:=\sup\left\{p\geq 0:\sigma_1(\mathbb{H}^{p+1}\times \mathbb S^{n-p-1}),\ldots,\sigma_{k}(\mathbb{H}^{p+1}\times \mathbb S^{n-p-1})>0\right\}.\end{equation*}
Exact formulas can be given for $k=2,3$. Indeed,
\begin{equation}\label{exact-formulas}\PP_2=\frac{n-\sqrt n-2}{2},\quad \PP_3=\frac{n-2-\sqrt{3n-2}}{2}.\end{equation}
Moreover, for 
fixed $k>1$, we have the asymptotic bound \cite{non-removable}
\begin{equation}\label{asymptotic-bound}
\frac{n}{2}-C_1(k)\sqrt n \leq \PP_k<\frac{n}{2}-\frac{2+\sqrt n}{2}\quad\text{as}\quad  n\to\infty \end{equation}
for some constant $C_1(k)$.

For general singular sets we also expect restrictions on $p$. The semilinear case, that is, for  $k=1$, is already well known in the literature. In addition to  the aforementioned reference Mazzeo-Pacard \cite{Mazzeo-Pacard}, we underline the classical result of Schoen-Yau \cite{Schoen-Yau:paper}, where they show that for complete, conformal metrics on $\mathbb S^n\setminus \Lambda$ of positive scalar curvature,  the Hausdorff dimension of the singular set $\Lambda$ must satisfy $\HH(\Lambda)\leq \frac{n-2}{2}$. We cite also Mazzeo-Smale \cite{Mazzeo-Smale}, where they consider singular sets on the sphere that are sufficiently close to an equatorial subsphere, and Fakhi \cite{Fakhi} when the singular set $\Lambda$ is a submanifold with boundary.

In the fully non-linear setting $k>1$, a necessary condition for the existence of complete metrics on $\mathbb S^n\setminus\Lambda$ conformal to the standard one was given by one of the authors in  \cite{non-removable} (see also \cite{Chang-Hang-Yang} for $k=2$). More precisely, if
$$\sigma_{1}(B_{g_u})\geq C_0>0\quad\mbox{and}\quad\sigma_2(B_{g_u}),\ldots,\sigma_k(B_{g_u})\geq 0$$
for some integer $1\leq k < n/2$, then
$$\HH(\Lambda)\leq\frac{n-2k}{2}.$$
This number is far from \eqref{asymptotic-bound} and we believe it can be improved. Indeed, we conjecture that a necessary condition for the existence of solutions to the $\sigma_k$--Yamabe problem in the positive  $\Gamma_k^+$ cone with singular set a $p$--dimensional, compact, closed submanifold  in $\mathbb S^n$ is
$p<\PP_k.$\\

\textbf{Scheme of the proof.} The main idea to find a solution of \eqref{eq} is to study its linearization, which is defined as follows. 
First, we fix $u>0$ and take a conformal perturbation of the metric $g_u$, i.e, we set for $s\in\mathbb R$,
$$s\mapsto g_s:=(u+s\varphi)^{\frac{8}{n-4}}g.$$
Using the normalization \eqref{tensor-B}, let $\mathbb B_s$ be the symmetric $(1,1)$-tensor given by
\begin{equation*}
\mathbb B_s:=\tfrac{n-4}{4} (u+s\varphi)^{\frac{2n}{n-4}} g_s^{-1} A_{g_s}.
\end{equation*}
The linearized operator of $\mathcal{N} (\,\cdot \,, g)$ about $u$ is defined as
\begin{equation}
\label{linope}
\begin{split}
\mathbb{L}(u,g)\,[\varphi]  :&=  \left.
\frac{d}{ds} \right|_{s=0} \mathcal{N}\, (u  + s\varphi,g)\\
&=\left.\frac{d}{ds}\right|_{s=0}\sigma_2(\mathbb B_s)-cq u^{q-1}\varphi.
\end{split}
\end{equation}
We can use a well known formula (see \cite{Viaclovsky:contact}) for the linearization of $\sigma_k$
\begin{equation}\label{calculation-linearized2}
\left.\frac{d}{ds}\right|_{s=0} \sigma_k(\mathbb B_s)=\trace\left( T^{k-1}(B_{g_u}) \left.\frac{d\mathbb B_s}{ds}\right|_{s=0}\right),
\end{equation}
where  $m$-th Newton tensor for a (1,1)-tensor $\mathbb B$ is given by
$$T^m(\mathbb B):=\sigma_m I-\sigma_{m-1} \mathbb B+\ldots+(-1)^{m} \mathbb B^m$$
for any integer $0\leq m\leq k$.
Note that, for a metric in the positive  $\Gamma_k^+$ cone, its Newton tensor $T^m$ is positive definite for $m=0,\ldots,k-1$. In particular, this implies that the linearization is elliptic. 

In addition, the non-degeneracy hypothesis means that the linearization around the background metric $g$, given by $\mathbb{L}(1,g)$, is invertible in $\mathcal C^\alpha(M)$.\\

As we have mentioned, the main idea   in the proof of Theorem \ref{main-theorem} is  to use a version of the gluing method from Mazzeo-Pacard \cite{Mazzeo-Pacard} in order to find solutions to \eqref{eq}. The starting point  is to construct an approximate metric, written  in the form
$$\bar g_\ep:=\bar u_\e^{\frac{8}{n-4}} g_M,\quad \bar u_\ep>0,$$ 
  with the right asymptotic behavior near the singularity (recall \eqref{asymptotic-behavior}), and then find a perturbation $\varphi$ such that
\begin{equation}\label{eq1}
\mathcal N(\bar u_\epsilon+\varphi,g_M)=0,
\end{equation}
by a linearization argument.
For simplicity, once $\bar u_\epsilon$ and $g_M$ have been fixed we will simply write
$$\mathbb L_\epsilon[\varphi]:=\mathbb L(\bar u_\epsilon,g_M)[\varphi].$$
Then, equation \eqref{eq1} is equivalent to
\begin{equation}\label{final-equation}
\mathbb{L}_\ep[\varphi] +f_\ep + Q_\ep[\varphi]=0,
\end{equation}
where we have defined
\begin{align}
&f_\ep:=\mathcal N(\bar u_\ep,g_M)\label{f-epsilon},\\
&Q_\ep[\varphi]:= \mathcal N(\bar u_\ep+\varphi,g_M)-\mathcal N(\bar u_\ep,g_M)-\mathbb L_\ep[\varphi].
\end{align}
Most of the analysis in this paper is concerned with the study of the mapping properties of the linearized operator in weighted spaces. Note that $\mathbb L_\ep$ is an edge operator in the sense of Mazzeo \cite{Mazzeo:edge}. The crucial point is that, near the singular set, $\mathbb L_\ep$ is modeled by an operator with constant indicial roots, that do not depend on $\ep$.
 Now,  if $\mathbb L_\ep$ has a right inverse, we can write \eqref{final-equation} as
\begin{equation*}
\varphi=-\mathbb L_\ep^{-1}(Q_\ep[\varphi]+f_\ep).
\end{equation*}
A fixed point argument will show the existence of a solution $\varphi$.

We remark that our arguments rely on the  good conformal properties of the equation.  Indeed, if two metrics $g_1$ and ${g}$ are related by 
$$u_1^{\frac{8}{n-4}} g_1=u^{\frac{8}{n-4}}{g},$$ then the non-linear operator from \eqref{eq} enjoys the following {\em conformal equivariance property}
\begin{equation}\label{confequi}
\mathcal{N} \, (u_1,g_1) =  (u/u_1)^{-\frac{4n}{n-4}} \, \mathcal{N} \, (u,g).%
\end{equation}
As a direct consequence, we have another conformal equivariance property for the linearized operator
\begin{equation}\label{confequilin}
\mathbb{L}(u_1,g_1)[\varphi]\,\,=\,\,(u/u_1)^{{-\frac{4n}{n-4}}}\,
\mathbb{L}(u,g)\left[(u/u_1)\, \varphi\right].
\end{equation}

\vskip 0.5cm

\textbf{Notations.}
Our reasoning will involve a delicate choice of real numbers $\mu$, $\nu$ and $\delta$. For future reference, in Figure 1 below we represent a real line with our choice of parameters. The marking points $\chi _{j,\pm}^{( 0)}$, $j=0,1$, are  indicial roots of the model linearization and will be specified in Subsection \ref{subsction:model-linearization}.

\begin{figure}[h]
\centering
\tikzset{every picture/.style={line width=0.75pt}} 

\begin{tikzpicture}[x=0.5pt,y=0.5pt,yscale=-1,xscale=1]

\draw [line width=2.25]    (19.33,111.67) -- (624.33,109.67) ;
\draw    (58.33,105.67) -- (58.33,120.67) ;
\draw    (313.33,103.67) -- (313.33,119.67) ;
\draw    (590.33,105.67) -- (590.33,120.67) ;
\draw    (430.33,104.67) -- (430.33,119.67) ;
\draw    (197.33,104.67) -- (197.33,119.67) ;
\draw    (129.33,103.67) -- (129.33,118.67) ;
\draw    (160.33,75.67) -- (160.02,105) ;
\draw [shift={(160,107)}, rotate = 270.61] [color={rgb, 255:red, 0; green, 0; blue, 0 }  ][line width=0.75]    (10.93,-3.29) .. controls (6.95,-1.4) and (3.31,-0.3) .. (0,0) .. controls (3.31,0.3) and (6.95,1.4) .. (10.93,3.29)   ;
\draw    (474.33,77.67) -- (474.95,105) ;
\draw [shift={(475,107)}, rotate = 268.7] [color={rgb, 255:red, 0; green, 0; blue, 0 }  ][line width=0.75]    (10.93,-3.29) .. controls (6.95,-1.4) and (3.31,-0.3) .. (0,0) .. controls (3.31,0.3) and (6.95,1.4) .. (10.93,3.29)   ;

\draw (286,123.4) node [anchor=north west][inner sep=0.75pt]  [font=\small]  {$\frac{p}{2} -\frac{n-4}{4}$};
\draw (44,132.4) node [anchor=north west][inner sep=0.75pt]    {$\chi _{1,-}^{( 0)}$};
\draw (420,130.4) node [anchor=north west][inner sep=0.75pt]    {$\chi _{0,+}^{( 0)}$};
\draw (576,132.4) node [anchor=north west][inner sep=0.75pt]    {$\chi _{1,+}^{( 0)}$};
\draw (184,131.4) node [anchor=north west][inner sep=0.75pt]    {$\chi _{0,-}^{( 0)}$};
\draw (465,43.4) node [anchor=north west][inner sep=0.75pt]  [font=\small]  {$\mu =\delta +\frac{p}{2} -\frac{n-4}{4}$};
\draw (155,39.4) node [anchor=north west][inner sep=0.75pt]  [font=\small]  {$\nu \approx -\delta +\frac{p}{2} -\frac{n-4}{4}$};
\draw (98,125.4) node [anchor=north west][inner sep=0.75pt]  [font=\small]  {$-\frac{n-4}{4}$};

\end{tikzpicture}
\label{figure4}
\caption{Our choice of parameters $\mu$ and $\nu$.}
\end{figure}
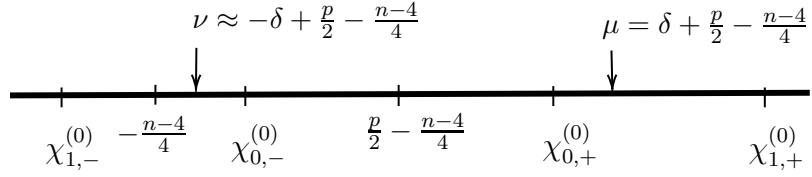

We conclude the Introduction with some notation remarks.
Denote the codimension of $\Lambda$ by
$$N=n-p.$$
 For one-variable functions $\varphi_1$,$\varphi_2$, we will write
$$\varphi_1\sim \varphi_2\quad\text{iff} \quad 0<c_1\leq \lim \frac{\varphi_1}{\varphi_2}\leq c_2$$
 and
 $$\varphi_1 \asymp \varphi_2 \quad\text{iff} \quad \lim \frac{\varphi_1}{\varphi_2}=1.$$
The convention $\varphi=O(r^\beta)$ means, not only an estimate for the function $\varphi$, but also for its derivatives (with the corresponding orders).
Finally, $C$ will denote a positive constant that may change from line to line, but always independently of $\ep$.\\

\textbf{Acknowledgements.} The authors would like to thank Rafe Mazzeo, Lorenzo Mazzieri, Frank Pacard and Mariel S\'aez for their useful advice, and to the anonymous referees for their valuable comments.\\

\section{The model $\mathbb R^n\setminus \mathbb R^p$ in cylindrical coordinates}

In this Section we consider the model  $\R^n \setminus \R^p$, which is equivalent to $\mathbb S^n\setminus\mathbb S^p$ by stereographic projection. It is the building block in our construction. First note that $\R^n \setminus \R^p$ is diffeomorphic to $\R_t \times \mathbb S_\theta^{N-1}\times \R_z^p$, so
we may write the Euclidean metric in polar-Fermi coordinates
\begin{equation}\label{metric-gE}
g_E :=dr^2+r^2g_\theta\,+ \,\delta_{\a \b} \,\, dz^\a \otimes \, dz^\b,
\end{equation}
where $\a,\b = 1, \ldots , p$  will refer to coordinates $z\in\mathbb R^p$,  and $g_\theta$ is the canonical metric on the sphere $\mathbb S^{N-1}$. Subindexes $i,j$ will refer to coordinates $\theta\in\mathbb S^{N-1}$.

We describe now the spherical harmonic decomposition of $\mathbb S^{N-1}$.
Let $\{\lambda_j\}_{j=0}^\infty$ be the eigenvalues for $-\Delta_{\theta}$ repeated according to multiplicity, with eigenfunctions $\{e_j(\theta)\}$, that is,
 \begin{equation}\label{spherical-harmonics}
 -\Delta_\theta e_j=\lambda_j  e_j, \quad j=0,1,\ldots
 \end{equation}
 In particular,
  \begin{equation*}
 \lambda_0=0,\quad\lambda_j=N-1 \quad\text{for }j=1,\ldots,N,\quad \text{and so on.}\\
 \end{equation*}

\subsection{Cylindrical coordinates}\label{subsection:cyl-coordinates}
Let us allow any $1\leq k<\frac{n}{2}$ to underline the general structure.
Our first observation is that it is more convenient to use a (conformal) cylinder-type metric as background.  For this, we set
\begin{equation*}
g_{cyl}=dt^2+g_\theta+e^{2t}\delta_{\a \b} \,\, dz^\a \otimes \, dz^\b.
\end{equation*}
Then, denoting the radial coordinate by $r=e^{-t}$, we see that $g_E$ and $g_{cyl}$ are conformally related by
\begin{equation*}
\begin{split}
g_E=   e^{-2t}  \big( \,  dt^2 \, + \, g_{\theta} \big)\,+ \,\delta_{\a \b} \,\, dz^\a \otimes \, dz^\b=e^{-2t}g_{cyl}.
\end{split}
\end{equation*}
We also change the notation for the conformal factor from $u$ to $v$ where
\begin{equation}\label{notation-uv}
u(r,\theta,z)=r^{-\frac{n-2k}{2k}}v(-\log r,\theta,z).
\end{equation}
Thus we can record any conformal change by
\begin{equation}\label{change:uv}
g_u = u^{\frac{4k}{n-2k}} g_{E}=v^{\frac{4k}{n-2k}} g_{cyl}.
\end{equation}
 With some abuse of notation, we will denote this conformal metric as
\begin{equation*}\label{conformal-metric}
g_v:=v^{\frac{4k}{n-2k}} g_{cyl}
\end{equation*}
for a conformal factor $v>0$.

Cylindrical coordinates provide a convenient framework for calculations.
The Schouten tensor for $g_{cyl}$ is diagonal and
\begin{align*}
&\big(g_{cyl}^{-1} A_{g_{cyl}}\big)_t^t= -\tfrac{1}{2},\\
&\big(g_{cyl}^{-1} A_{g_{cyl}}\big)_j^i= +\tfrac{1}{2}\,\delta_j^i,\\
&\big(g_{cyl}^{-1} A_{g_{cyl}}\big)_\beta^\alpha= -\tfrac{1}{2}\,\delta_\beta^\alpha.
\end{align*}
Define
\begin{equation*}
  B_{g_v}=\tfrac{n-2k}{2k}\, v^{\frac{2n}{n-2k}}\,g_{v}^{-1}A_{g_{cyl}},\quad  B_{g_{cyl}}=\tfrac{n-2k}{2k}\, g_{cyl}^{-1}A_{g_{cyl}}.
\end{equation*}
Then, since
\begin{equation*}
A_{g_v}=A_{g_{cyl}}-\tfrac{2k}{n-2k} \,v^{-1} D^2 v+\tfrac{2kn}{(n-2k)^2}\,v^{-2} \nabla v\otimes\nabla v-\tfrac{2k^2}{(n-2k)^2}\,v^{-2}\,|\nabla v|^2_{g_{cyl}}g_{cyl},
\end{equation*}
where all the derivatives and norms are calculated with respect to the background metric $g_{cyl}$, we have
\begin{equation}\label{B-tensor-v}
B_{g_v}=g_{cyl}^{-1}\left[v^2B_{g_{cyl}}-v D^2 v+\tfrac{n}{n-2k}\, \nabla v\otimes\nabla v-\tfrac{k}{n-2k}\,|\nabla v|^2_{g_{cyl}}g_{cyl}\right],
\end{equation}
and thus
\begin{equation}\label{B-coordinates}\begin{split}
\big( B_{g_v}\big)_t^t & =    -\tfrac{n-2k}{4k} \,v^2  - v\partial_{tt}v +\tfrac{n}{n-2k}\,(\partial_t v)^2-\tfrac{k}{n-2k}\,|\nabla v|^2_{g_{cyl}}, \\
\big( B_{g_v}\big)_j^i & =  \tfrac{n-2k}{4k} \,v^2 \,\delta_{j}^i-\tfrac{k}{n-2k}\, |\nabla v|^2_{g_{cyl}} \,\delta_{j}^i,\\
&\quad+\left[-v\,(D^2_\theta v)_{lj}+\tfrac{n}{n-2k}\,\partial_l v\,\partial_j v\right](g_\theta)^{il},\\
\big( B_{g_v}\big)_\beta^\alpha & =  \left[-\tfrac{n-2k}{4k}\,v^2-v\,\partial_t v-\tfrac{k}{n-2k}\,|\nabla v|^2_{g_{cyl}}\right]\delta_\beta^\alpha\\&\quad+\left[-v\,\partial_{\alpha\beta} v+\tfrac{n}{n-2k}\,\partial_\alpha v\,\partial_\beta v\right]e^{-2t}.
\end{split}\end{equation}
Notice that, for a function $v=v(t)$, the matrix $B_{g_v}$ is diagonal and
\begin{equation}\label{B-radial}\begin{split}
\big( B_{g_v}\big)_t^t & =    -\tfrac{n-2k}{4k} \,v^2  - v\ddot v +\tfrac{n-k}{n-2k}\,\dot v^2=:\kappa_1,\\
\big( B_{g_v}\big)_j^i & =  \big[\tfrac{n-2k}{4k} \,v^2 -\tfrac{k}{n-2k}\, \dot v^2\big]\delta_{j}^i=:\kappa_2\,\delta_{j}^i,\\
\big( B_{g_v}\big)_\beta^\alpha & =  \big[-\tfrac{n-2k}{4k}\,v^2-v\,\dot v-\tfrac{k}{n-2k}\,\dot v^2\big]\delta_\beta^\alpha=:\kappa_3\,\delta_\beta^\alpha.
\end{split}\end{equation}
Its eigenvalues are $\kappa_1,\kappa_2,\kappa_3$ with multiplicities $1,N-1,p$, respectively.\\

Now, the covariance property \eqref{confequi} implies that
\begin{equation*}\label{covariance-r}
\mathcal N(u,g_E)=r^n \mathcal N(v,g_{cyl}).
\end{equation*}
Thus  we can rewrite  the original equation \eqref{eq} in this notation  as
\begin{equation}\label{equation-v}
  \mathcal N(v,g_{cyl})=0.
\end{equation}

\subsection{The fast-decay ODE solution}

In the paper \cite{Gonzalez-Mazzieri} the authors construct a very particular solution $U_1=U_1(r)$ for our problem \eqref{eq} on $\mathbb R^n\setminus\mathbb R^p$ that yields a complete metric near the singular set $\{r=0\}$ and has fast decay as $r\to\infty$. Here is the first place where we encounter the restriction $k=2$. More precisely:

\begin{pro}\cite{Gonzalez-Mazzieri}\label{ODE-study}
For each $0<p<\PP_2$, there exists a positive solution $U_1=U_1(r)$ for the equation
\begin{equation}\label{ODE-sigma2}
\sigma_2  \left( B_{g_{u}} \right)  = c u^{q}\quad\text{in}\quad\mathbb R^n\setminus\mathbb R^p,
\end{equation}
 satisfying:
\begin{itemize}
\item When $r\to 0$, the solution has the precise asymptotic behavior
$$U_1(r) \asymp v_\infty r^{-\frac{n-4}{4}}$$
where
\begin{equation}\label{v-infty}
(v_\infty)^{\frac{16}{n-4}}=c_{n,p,2}{\binom{n}{2}}^{-1}>0.
\end{equation}
and the constant $c_{n,p,2}$ is defined in \eqref{model-sigma-k}.

\item When $r\to +\infty$,
$$U_1(r) \asymp C r^{-\alpha_0-\frac{n-4}{4}}$$
for some $\alpha_0\in\lp 0,\frac{n-4}{4}\rp$, and some positive constant $C$.

\item $r^{\frac{n-4}{4}}U_1 $ is uniformly bounded for all $r>0$.

\item The metric $g_{U_1}:=(U_1)^{\frac{8}{n-4}} g_{E}$ belongs to the positive  $\Gamma_2^+$ cone.
\end{itemize}
\end{pro}

\begin{proof}[Sketch of the proof]
Although we will not give the full proof of this result, which is contained in \cite{Gonzalez-Mazzieri}, the key idea is to use the framework in cylindrical coordinates we have just presented and to find a solution $V_1$ to \eqref{equation-v} that only depends on the radial variable $t$.

From the formulas in \eqref{B-radial} one is able to write $\sigma_2(B_{g_v})$ in a reasonably simple form. Standard phase-plane analysis for the ODE \eqref{equation-v} completes the proof.  Now  we recover $U_1$ by setting, in the notation \eqref{notation-uv},
$$U_1(r)= r^{-\frac{n-4}{4}}V_1(-\log r).$$
\end{proof}

\begin{cor}
\label{cor-model-solution}
For $\ep>0$ we rescale
\begin{equation}\label{rescale}U_\ep(r)=\ep^{-\frac{n-4}{4}}U_1\lp\frac{r}{\ep} \rp. \end{equation}
Then
 $U_\ep$ is a solution of \eqref{eq} in $\mathbb R^n\setminus \mathbb R^p$ that only depends on the radial variable. Moreover, for $m$ large enough (independent of $\ep$),
\begin{equation}\label{behavior-zero}U_\ep(r)\asymp v_\infty \, r^{-\frac{n-4}{4}},\quad\text{when}\quad r\in\left(0,\tfrac{1}{m}\ep\right),\end{equation}
\begin{equation}\label{behavior-infty}
  U_\ep(r)\asymp C\ep^{\alpha_0} r^{-\alpha_0-\frac{n-4}{4}}\quad\text{when}\quad r\geq m\ep.
\end{equation}
\end{cor}

In the following, we will set
$$\alpha_1:=\alpha_0+\tfrac{n-4}{4},$$
$$U_\ep(r)=r^{-\frac{n-4}{4}}V_\ep(-\log r),$$
for  $V_\ep=V_\ep(t)$, 
and
\begin{equation}\label{metric-v-epsilon}
g_{V_\ep} = V_\ep^{\frac{8}{n-4}} g_{cyl}.
\end{equation}
Note that 
\begin{equation}\label{V1-Vep}
V_\ep(t)=V_1(-\log r +\log \ep).
\end{equation}

We will consider the linearized operator in both coordinate systems,
\begin{equation}\label{two-linearizations}
\mathcal L_\ep :=\mathbb L(U_\ep,g_E) \quad\text{and}\quad L_\ep:=\mathbb L(V_\ep,g_{cyl}).
\end{equation}
These are our model operators. Their relation is given by the covariance property \eqref{confequilin}, more precisely,
\begin{equation}\label{rewrite}
\begin{split}
\mathcal L_\ep \varphi=r^n L_\ep [r^{\frac{n-4}{4}}\varphi].
\end{split}
\end{equation}
For convenience, we will also set 
$$w=r^{\frac{n-4}{4}} \varphi.$$
We observe that 
the operator $L_\ep$ has better mapping properties in weighted Sobolev spaces that we will explain in Section \ref{section:function-spaces}.


\section{The approximate solution}\label{section:approximate-solution}

Let $\Lambda$ be a smooth, closed submanifold in $M$  of dimension $p$, such that $0<p<\PP_2$. Here we construct an approximate solution $\bar u_\epsilon>0$ to problem \eqref{eq} on $M\setminus \Lambda$ which is singular exactly at $\Lambda$ with a precise blowup rate and, in addition, remains in the positive  $\Gamma_2^+$ cone.

Let $\mathcal{T}_\rho$ be the tubular neighborhood of radius $\rho$ around $\Lambda$. It is well known that $\mathcal{T}_\rho$ is a disk bundle over $\Lambda$; more precisely, it is diffeomorphic to the bundle of radius $\rho$ in the normal bundle $\mathcal N\Lambda$.  The Fermi coordinates will be constructed as coordinates in the normal bundle transferred to $\mathcal{T}_\rho$ via such diffeomorphism. More precisely, let $r$ be the distance to $\Lambda$, which is well defined and smooth away from $\Lambda$ for small $\rho$. Let also $z$ be a local coordinate system on $\Lambda$ and $\theta$ the angular variable on the sphere in each normal space $\mathcal N_z\Lambda$. We denote by $\mathcal B_{\rho}$ the ball of radius $\rho$ in $\mathcal N\Lambda$ at a point $z\in\Lambda$. Finally we let $x$ denote the rectangular coordinate in these normal spaces, so that $r=|x|$, $\theta=\frac{x}{|x|}$. Thus we can identify $\mathcal T_\rho$ with $\mathcal B_\rho\times\Lambda$, parameterized with coordinates $y=(x,z)$, $x\in \mathcal B_\rho$, $z\in \Lambda$.

In the classical paper of Mazzeo-Pacard \cite{Mazzeo-Pacard}, the global approximate solution is constructed from  the model solution $U_\epsilon$ given in \eqref{rescale}, and extended to zero with the introduction of a  cutoff in a ball of radius $r_0>>\epsilon$. This does not work for $\sigma_2$ since such approximate solution may not be in the  positive $\Gamma_2^+$ cone. Instead, we need to make a more refined choice of this cutoff  based on Lemma 7 from Guan-Wang \cite{Guan-Wang:inequalities}, where the authors are able to transplant an isolated singularity with given asymptotics to any conformally flat metric, while still remaining in the positive $\Gamma_2^+$ cone.
 We remark that this lemma is the key ingredient in the construction from  \cite{Guan-Lin-Wang} of manifolds of positive $\sigma_k$--curvature by connected sums of compact manifolds.

More precisely, we will show:

\begin{pro}\label{prop-gluing-Lambda}
 Let $g_0=u_0^{\frac{8}{n-4}}g_{M}$ be a smooth metric on a tubular neighborhood $\mathcal{T}_\rho$ of $\Lambda$ in $M$ satisfying $g_0\in \Gamma_2^+$. Fix  any $0<\alpha_1<\frac{n-4}{2}$. Then there exists a conformal metric $$g=u^{\frac{8}{n-4}}g_{M}\quad\text{on}\quad\mathcal T_\rho\setminus\{r=0\}$$
 that belongs to the positive $\Gamma_2^+$ cone  and satisfies the following:
\begin{align}
 & u(r,\theta,z) =u_0(r,\theta,z) \mbox{ in }\mathcal T_\rho\setminus \mathcal T_{\varrho_1}, \label{behavior-near}\\
  &u(r,\theta,z)=r^{-\alpha_1} \mbox{ in }\mathcal T_{\varrho}\setminus\{r=0\},\label{behavior-far}
\end{align}
for some $0<\varrho<\varrho_1<\rho$.
In addition, if $g_0$ is non-degenerate, then $g$ is also non-degenerate.
\end{pro}

The proof will follow in two steps: first we consider the model case  $\mathbb R^n\setminus \mathbb R^p$ (Subsection \ref{subsection:model-singularity}), and then transplant this construction to $M\setminus\Lambda$ using Fermi coordinates (Subsection \ref{subsection:Fermi-coordinates}).

\subsection{The model singularity}\label{subsection:model-singularity}

Assume that  we are in the flat model
$$\mathbb R^n\setminus \mathbb R^p=\mathbb R^+_r\times \mathbb S_\theta^{N-1}\times \mathbb R_z^p,$$
with the Euclidean metric in polar-Fermi coordinates $g_E$, given by \eqref{metric-gE}. 

We will show that, given a  metric of the form $g_0=u_0^{\frac{8}{n-4}}|dx|^2$ in a tubular neighborhood  $\mathcal T_\rho(\mathbb R^p)$ in $\mathbb R^n$, it is possible to construct a conformal metric that is singular exactly on $\mathbb R^p$. Without loss of generality, we take $\rho=1$. The main idea is to adapt the proof in \cite{Guan-Wang:inequalities} for isolated singularities, which is possible since the radial variable $r$ plays the same role. Nevertheless, the difference here is that  we have an additional coordinate $z\in\mathbb R^p$, so there is an extra block $J_1$ in all the matrices (see formula \eqref{matrix-J} below).


We first change the notation for the conformal factor, writing
\begin{equation}\label{change-metric}
g=e^{-2\omega}g_0=e^{-2(\omega+\omega_0)}g_E=u^{\frac{8}{n-4}}g_E,
\end{equation}
 where we have set $\omega_0=-\frac{4}{n-4}\log u_0$. The
 transformation law for the Schouten tensor is
 \begin{equation}
 A_g=D^2 \omega+\nabla\omega\otimes \nabla\omega-\tfrac{1}{2}|\nabla \omega|^2_{g_0} g_0+A_{g_0}.
 \end{equation}
 In this particular case, using  the Euclidean metric as background,
\begin{equation}\label{Schouten-conformal-omega}
\begin{split}
  A_{g}&= D^2(\omega+\omega_0)+\nabla (\omega+\omega_0)\otimes \nabla(\omega+\omega_0)-\tfrac{1}{2}|\nabla (\omega+\omega_0)|^2I\\
&= D^2 \omega +\nabla \omega\otimes\nabla \omega+\nabla \omega\otimes \nabla \omega_0+\nabla \omega_0\otimes\nabla \omega\\&
\quad-\lp\tfrac{1}{2}|\nabla \omega|^2+ \langle \nabla \omega,\nabla \omega_0\rangle\rp I+A_{g_0},
\end{split}
\end{equation}
 where all the derivatives are taken with respect to the Euclidean metric on $\mathbb R^n$.

In the gluing process we will  study carefully the neck region. In order to match the behaviors \eqref{behavior-near}--\eqref{behavior-far},  we seek $\omega$ so that
\begin{equation*}
\left\{\begin{split}
  &\omega=0 \quad\text{near}\quad r=1, \\
  &\omega=\alpha_2 \log r-\omega_0 \quad\text{near}\quad  r=0,
\end{split}\right.
\end{equation*}
and such that $A_{  g}$ remains in the positive  $\Gamma_2^+$ cone. Here we have taken $\alpha_2= \tfrac{4}{n-4}\,\alpha_1$.
As a first approximation, we impose $\omega$ to be of the form
 \begin{equation}\label{omega-prime}
   \omega'(r)=\frac{\phi(r)}{r},
 \end{equation} 
for some suitable transition function $\phi(r)$ satisfying
\begin{equation*}
\left\{\begin{split}
  &\phi( r)=0 \quad\text{near}\quad r=1, \\
  &\phi( r)=\alpha_2 \quad\text{near}\quad  r=0.
\end{split}\right.
\end{equation*}
By straightforward calculation from  \eqref{Schouten-conformal-omega} we have that
\begin{equation}\label{formula1000}
A_{  g}=A_{g_0}+ J
+E(\phi),
\end{equation}
where
\begin{equation}\label{matrix-J}
 J=\begin{bmatrix}
J_0 & 0\\
0&  J_1
\end{bmatrix}
\end{equation}
for
\begin{equation*}
J_0=\frac{2\phi-\phi^2}{2 r^2}\,I_{N\times N}+\lp \frac{\phi'}{ r}+\frac{\phi^2-2\phi}{ r^2}\rp \theta\otimes \theta \quad\text{and}\quad   J_1=-\frac{\phi^2}{2r^2}I_{p\times p},
\end{equation*}
and $E$ is the perturbation term. Now,
$$\sigma_2(  g^{-1} A_{  g})=e^{-4(  \omega+\omega_0)} \sigma_2(g_E^{-1} A_{  g}).$$

Let us take $r_0$ be a small enough (to be specified later), and  set $\phi(r)=0$ for $r\in[r_0,1)$, so that $g=g_0$.  To extend $\phi$ to the whole interval $(0,r_0)$ we proceed as follows. Since $A_{g_0}$ belongs to  the positive  $\Gamma_2^+$ cone and is non-degenerate, which are open conditions, we can take $r_1\in(0,r_0)$ and $\phi:[r_1,r_0)\to[0,\alpha_1)$ non-increasing such that $A_{g}$ also belongs to the positive $\Gamma_2^+$ cone, is non-degenerate, and $\phi(r_1)>0$  (but very small).

Let us consider first the unperturbed matrix $J$. One easily calculates
\begin{equation*}
  \sigma_2(J_0)=\frac{N-1}{2}\lp\frac{2\phi-\phi^2}{2r^2}\rp^2
  \left[N-4+4\frac{r\phi'}{2\phi-\phi^2}\right]
\end{equation*}
and
\begin{equation*}
  \sigma_1(J_0)=\frac{2\phi-\phi^2}{2r^2}
  \left[N-2+2\frac{r\phi'}{2\phi-\phi^2}\right].
\end{equation*}
We set  $\phi$ to be a solution of the ODE
\begin{equation*}
4\frac{r\phi'}{2\phi-\phi^2}=-\frac{1}{2},
\end{equation*}
with initial condition given by the above value for $\phi(r_1)$.
This ODE is easily integrated, which yields,
\begin{equation*}
  \phi(r)=\frac{2\delta}{\delta+r^{\frac{1}{4}}},
\end{equation*}
for some small constant $\delta>0$ determined by the given $\phi(r_1)$. 

With this choice of $\phi$ we can calculate the eigenvalues of  the full matrix $J$,
\begin{equation}\label{eigenvalues-pp}
  \varsigma_1=-\frac{5}{8}\,\frac{2\phi-\phi^2}{r^2},\quad  \varsigma_2=\frac{2\phi-\phi^2}{2r^2},\quad \varsigma_3=-\frac{\phi^2}{2r^2},
\end{equation}
with multiplicities $1$, $N-1$ and $p$, respectively.  After some tedious but straightforward calculation, one has that 
\begin{equation*}
\begin{split}
   \sigma_2( J)&=(N-1)\varsigma_1\varsigma_2+p\varsigma_1\varsigma_3+\binom{N-1}{2}\varsigma_2^2+\binom{p}{2}\varsigma_3^2
   \\&+p(N-1)\varsigma_2\varsigma_3>0, \\
\sigma_1(J)&=\varsigma_1+(N-1)\varsigma_2+p\varsigma_3>0,
\end{split}
\end{equation*}
which means that $J$ is in the positive  cone $\Gamma_2^+$.

Moreover, the  error term $E(\phi)$ can be estimated  by
\begin{equation}\label{Ephi}|E(\phi)|\leq O(|\nabla \omega_0|)\frac{\phi}{ r}\leq \tau \frac{\phi}{r^2}
\end{equation}
for some $\tau$ small enough but depending only on $n,p,\omega_0$, thus it is only a perturbation of the eigenvalues
\eqref{eigenvalues-pp} for $r$ small. For this, here we choose the initial $r_0$ small enough, depending only on $\omega_0$.

Since the matrix $J$ belongs to the positive $\Gamma_2^+$ cone, also does $A_{  g}$ in the region $0<r<r_1$.
 Then choose $r_2\in(0,r_1)$ such that $\phi(r_2)=\alpha_2$, which is possible because $0<\alpha_2<2$ by hypothesis and define $\phi(r) = \alpha_2$ on $[0,r_2]$. The first statement of  Proposition \ref{prop-gluing-Lambda} in the $\mathbb R^n\setminus \mathbb R^p$ setting follows by smoothing out the conformal metric $ g$, which is already $\mathcal C^{1,1}$ by construction. 
 
 For the non-degeneracy statement, let $\omega$ be conformal factor we have just constructed. Let us look closely at the linearized operator $\mathbb L (\omega,g_0)$ (in the notation \eqref{change-metric}). For this, recall the formulas \eqref{linope}-\eqref{calculation-linearized2}. Non-degeneracy for a related linearized operator will be studied in detail in Section \ref{section:injectivity}. Nevertheless, in our case we are working on a bounded region away from the singularity where everything is smooth, non-degeneracy is much simpler as we do not have singular points. 
 
Observe first the general formula \eqref{calculation-linearized2}. Since the Newton tensor $T^1(A_g)$ is positive definite, we just need to look at the different blocks coming from $\left.\frac{dA_g}{ds}\right|_{s=0}$ and, in particular, the variation of the matrix $J$. In the radial direction, non-degeneracy reduces to study the Wronskian of a  regular ODE in a bounded interval. Since it does not vanish near $r_1$ by the non-degeneracy hypothesis of $g_M$, then it is non-vanishing everywhere.

Along other directions, we need to take into account the other blocks of the matrix $J$. When $J$ acts on the angular direction, the only contribution to the zero-th order potential term in $\mathbb L(\omega,g_0)\varphi$ (which may spoil non-degeneracy) is $-\frac{1-\phi}{r^2}\varphi$, which is negative by construction. Similarly, $J_1$ part contributes with a term $-\frac{\phi}{r^2}\varphi$, which is also negative and makes the operator even more coercive.

\subsection{Fermi coordinates}\label{subsection:Fermi-coordinates}

Now we consider a general singular set $\Lambda$, which is taken to be a smooth compact, connected, closed, submanifold in $M$ of dimension $p$, and finish the proof of Proposition \ref{prop-gluing-Lambda}.

We use  Fermi coordinates in a tubular neighborhood  $\mathcal T_\rho$ around $\Lambda$. Any metric $g$ can be compared to the canonical metric $g_E$ from \eqref{metric-gE}. Indeed,
\begin{equation}\label{metric-Fermi}
g=
\begin{pmatrix}1&0& O(r)\\
0&r^2g_\theta+O(r^4)&O(r^2)\\
O(r)&O(r^2)&g_\Lambda+O(r)\end{pmatrix},
\end{equation}
where $g_\Lambda$ is the metric on $\Lambda$.
This expansion is classical, see, for instance, \cite{Finn-McOwen}.

Then, the only difference with the arguments in Section \ref{subsection:model-singularity} is to keep track of the extra perturbation terms.  More precisely, the background metric in Fermi coordinates (given in the formula \eqref{metric-Fermi}) near the singular set differs from the model metric  \eqref{metric-gE} in two ways:
\begin{itemize}
\item Perturbation terms of order (at least) $O(r)$ as $r\to 0$.
\item In the $z$ direction the flat metric $|dz|^2$ is replaced by $g_\Lambda$.
\end{itemize}
Similarly to the above, expression \eqref{formula1000} is rewritten as
\begin{equation*}
A_{g}=A_{g_0}+ J
+E(\phi),
\end{equation*}
where
\begin{equation*}
  J=\begin{bmatrix}
 J_0 & 0\\
0&  J_1
\end{bmatrix}
\end{equation*}
for
\begin{equation*}
   J_0=\frac{2\phi-\phi^2}{2 r^2}\,I_{N\times N}+\lp \frac{\phi'}{ r}+\frac{\phi^2-2\phi}{ r^2}\rp\theta\otimes \theta,\quad\quad   J_1=-\frac{\phi^2}{2r^2}g_\Lambda,
\end{equation*}
and $E$ is the perturbation term.

To handle the block $J_1$  we work in small enough neighborhoods of $z\in \Lambda$ and write $g_\Lambda$ in normal coordinates. Now, for the term $E(\phi)$ we claim that we still have that  $|E(\phi)|\leq \tau\frac{\phi}{r^2}$ for some small enough $\tau$. To see this,  one needs to control the $O(r)$ terms in the expansion  of the metric $g_{\mathbb R^n}$ in Fermi coordinates from \eqref{metric-Fermi}. For instance, for the inverse of the metric we have
\begin{equation*}
(g_{\mathbb R^n})^{-1}=
\begin{pmatrix}1+O(r)&O(r)& O(r)\\
O(r)&r^{-2}g_\theta+O(r^{-1})&O(1)\\
O(r)&O(1)&g_\Lambda^{-1}+O(r)\end{pmatrix},
\end{equation*}
the  Christoffel symbols of $g_{\mathbb R^n}$
\begin{equation*}
\begin{array}{lll}
\Gamma_{rr}^r= O(r),&\Gamma_{r\theta}^r=O(r^2),&  \Gamma^r_{rz}=O(r),  \\
\Gamma_{\theta \theta}^r=r g_\theta+O(r^2),&\Gamma_{\theta z}^r=O(r),&  \Gamma^r_{z z}=O(1),
\end{array}
\end{equation*}
and the Hessian of a function  $\omega=\omega(r)$
\begin{equation}\label{Hessian}
 D^2\omega=\begin{pmatrix}
\partial_{rr} \omega +O(r)\partial_r \omega & O(r^2)\partial_r \omega& O(r)\partial_r \omega\\
 O(r^2)\partial_r \omega & -[r g_\theta+O(r^2)]\partial_r \omega & O(r)\partial_r \omega\\
 O(r)\partial_r \omega & O(r)\partial_r \omega & O(1)\partial_r \omega
\end{pmatrix}.
\end{equation}
Therefore, the perturbation terms do not change the eigenvalues as $r\to 0$.

The proof of Proposition \ref{prop-gluing-Lambda} is completed. 

\qed

\subsection{Construction of $\bar u_\ep$}\label{subsection:approximate-sol}

Let $M$ be a $n$-dimensional, compact, smooth, Riemannian manifold of constant $\sigma_2$--curvature (and in the positive $\Gamma_2^+$ cone), in the hypothesis of Theorem \ref{main-theorem}. We would like to transplant the singular metric with  conformal factor $U_\ep(r)$ to this background metric, for $\ep$ small enough. 

We first use Proposition \ref{prop-gluing-Lambda} to construct a non-degenerate metric $g_*=u_*^{\frac{8}{n-4}}g_M$ on $M$ satisfying 
\begin{equation}\label{rescaling-metric}
u_* =\left\{\begin{split}
&r^{-\alpha_1}\quad \text{in }\mathcal T_{\varrho}(\Lambda)\setminus \Lambda,\\
& 1 \quad \text{outside }\mathcal T_{\varrho_1}(\Lambda),
\end{split}\right.
\end{equation}
in the positive $\Gamma_2^+$ cone.   Then rescale
 $$g_\ep=u_\ep^{\frac{8}{n-4}}g_*,\quad \text{for}\quad u_\ep:=\ep^{\alpha_0}u_*.$$
Looking at the asymptotic behavior from \eqref{behavior-infty}, it is clear that we can patch both functions $u_\ep$ and $U_\ep$ (as they have the same asymptotics in $r$ and $\ep$) in order to construct a globally defined metric
\begin{equation}\label{background-metric-u}
\bar g_{\eps}=\bar u_\ep^{\frac{8}{n-4}}g_M\quad\text{on }M\setminus \Lambda
\end{equation}
in the positive  $\Gamma_2^+$ cone. Here $\varrho,\varrho_1$ do not depend on $\ep$, and satisfy $\varrho,\varrho_1>>m\ep$ as $\ep \to 0$.


Finally, the behavior of our approximate solution $\bar u_\epsilon$ near the singular set  is given by
$$\bar u_\ep =U_\ep \asymp v_\infty r^{-\frac{n-4}{4}} \quad\text{when}\quad r< \tfrac{1}{m}\ep,$$
which agrees with \eqref{asymptotic-behavior}.

\subsection{The non-linear problem}\label{subsection:estimates}

Here we check that, indeed, $\bar u_\ep>0$ is a good approximate solution. For this, it is better to switch to the $v$-notation in Fermi coordinates using the cylindrical variable $t=-\log r$. Thus, with some abuse of notation, we set $g_{M,cyl}$ to be the background metric on $M$ but using cylindrical Fermi coordinates around $\Lambda$ and the normalization \eqref{change:uv}. The approximate solution $\bar u_\ep$ in this new setting will be denoted by $\bar v_\ep$, and the approximate metric \eqref{background-metric-u} by
\begin{equation}\label{background-metric-v}
\bar g_{\ep,cyl}=\bar v_\ep^{\frac{8}{n-4}}g_{M,cyl}\quad\text{on }M\setminus\Lambda.
\end{equation}
We would like to solve the problem
\begin{equation}
\label{eq-t}
  \mathcal{N} (v, g_{M,cyl})   := \sigma_2  \left( B_{g_{v}}
  \right)  - c |v|^{q-1}v  = 0.
\end{equation}
We denote by $\bar L_\ep$  the linearized operator in this notation, given by $\bar L_\ep:=\mathbb L(\bar v_\ep,g_{M,cyl})$.
Then we can rewrite \eqref{final-equation} in cylindrical coordinates as
\begin{equation}\label{eq-tt}
\bar{L}_\ep[w] +\bar f_\ep + \bar Q_\ep[w]=0,
\end{equation}
where we have defined
\begin{equation*}
\begin{split}
&\bar f_\ep:=\mathcal N(\bar v_\ep,g_{M,cyl}),\\
&\bar Q_\ep[w]:= \mathcal N(\bar v_\ep+w,g_{M,cyl})-\mathcal N(\bar v_\ep,g_{M,cyl})-\bar L_\ep[w].
\end{split}
\end{equation*}
We prove:

\begin{pro}\label{prop:compare-metric}
For the error term we have the estimate
\begin{equation*}
\bar f_\ep=\begin{cases}
O(r),&\quad r\in (0,\frac{1}{m}\ep),\\
O(\ep),&\quad r\in (\frac{1}{m}\ep,m\ep),\\
O(\ep^{4\alpha_0} r^{-4\alpha_0+1}),&\quad r\in(m\ep,\varrho),\\
O(\ep^{4\alpha_0}),&\quad r>\varrho.
\end{cases}
\end{equation*}
The constants may depend on $m$ but not on $\ep$.
\end{pro}

\begin{proof}
On the one hand, in the neighborhood $\mathcal T_\varrho$ we have defined $\bar v_\ep=V_\ep$, the model solution from Corollary \ref{cor-model-solution}. Observe that, in the model case  $\mathcal N(V_\ep, g_{cyl})\equiv 0$ from Proposition \ref{ODE-study}.
  In order to calculate the error $\mathcal N(V_\ep,g_{M,cyl})$ in this region we  compare the background metric $g_{M}$ in Fermi coordinates given in \eqref{metric-Fermi} to the model  $g_{E}$ from \eqref{metric-gE}  when $r\to 0$ as we did in the proof of Proposition \ref{prop-gluing-Lambda}
(one can also refer to the arguments in the proof of Proposition 2.19 of \cite{Mazzeo-Smale}).

On the other hand,  away from the singular set (say, $r>\varrho$), $\bar v_\ep$ is essentially a factor of $\ep^{\alpha_0}$ times a smooth function, thus we can calculate the error by keeping track of this rescaling in our
definition of the nonlinear operator.

With this we can conclude, with respect to the background metric $g_M$,
\begin{equation*}
B_{\bar g_v}=
\begin{cases}
B_{v_\infty,cyl}+O(r),&\quad r\in (0,\frac{1}{m}\ep),\\
B_{V_\ep,cyl}+O(r),&\quad r\in (\frac{1}{m}\ep,m\ep),\\
B_{V_1(-\log r+\log\ep),cyl}+O(V_\ep^2 r),&\quad r\in(m\ep,\varrho),\\
\ep^{2\alpha_0}B_{v_*,cyl}+O(v_*^2\,r),&\quad r\in(\varrho,\varrho_1),  \\
\ep^{2\alpha_0}B_{g_M},&\quad r>\varrho_1.
\end{cases}
\end{equation*}
Here we have denoted by $v_*$ the function $u_*$ with the normalization \eqref{change:uv}. Thus
\begin{equation*}
\sigma_2(B_{\bar g_v})=
\begin{cases}
cv_\infty^{q-1}+O(r),&\quad r\in (0,\frac{1}{m}\ep),\\
cV_\ep^{q-1}+O(r),&\quad r\in (\frac{1}{m}\ep,m\ep),\\
cV_1^{q}(-\log r+\log\ep)+O(V_\ep^4r),&\quad r\in(m\ep,\varrho),\\
O(\ep^{4\alpha_0}),&\quad r>\varrho,
\end{cases}
\end{equation*}
which yields the conclusion (recall the relation \eqref{V1-Vep}).
\end{proof}


\section{Explicit calculations in the model case}\label{section:explicit}

In this section we work with the model case $\mathbb R^n\setminus\mathbb R^p$ and the metric $g_{V_\ep}$ given in \eqref{metric-v-epsilon}, in order to study the model linearization $\mathcal L_\ep$ given in  \eqref{two-linearizations}. These will be needed in Section \ref{subsction:model-linearization} when we study the full linearized operator. 

We first find the general structure of $\mathcal L_\ep$. Since this part is valid for all $k<\frac{n}{2}$, we do not restrict to $k=2$ in Lemma \ref{lema:model-linearization} below.

Take now a conformal perturbation of the metric $g_{V_\ep}$ on $\mathbb R^n\setminus \mathbb R^p$, i.e, for $s\in\mathbb R$, set
$$s\mapsto g_s:=(V_\ep+sw)^{\frac{4k}{n-2k}}\left[ dt^2 +  g_{\theta} + e^{2t}\,\delta_{\a \b} \,dz^\a \otimes dz^\b\right],$$
and consider  $B_s$  the symmetric $(1,1)$-tensor given by
\begin{equation}\label{Bs}B_s:=\tfrac{n-2k}{2k} (V_\ep+sw)^{\frac{2n}{n-2k}} g_s^{-1}A_{g_s},\end{equation}
that is,
\begin{equation}\label{calculation-linearized1}
\begin{split}
L_\ep[w]:=\left.\frac{d}{ds}\right|_{s=0} \mathcal N (V_\ep+sw,g_{cyl})
=\left.\frac{d}{ds}\right|_{s=0}\sigma_k(B_s) -cq(V_\ep)^{q-1}w.
\end{split}\end{equation}

Our first result is a structure statement:

\begin{lem}\label{lema:model-linearization}
The linearized operator for the model in cylindrical coordinates is given by
\begin{equation}\label{L-epsilon}
   L_\ep[w]=a_0^\ep w + a_1^\ep \partial_t w +a_2^\ep\partial_{tt} w+a_3^\ep\Delta_\theta w+a_4^\ep e^{-2t}\Delta_z w,
\end{equation}
where the coefficient functions $a_\ell^\ep=a_\ell^\ep(t)$ are given in the proof below. Moreover, $a^\ep_2,a^\ep_3,a^\ep_4$ have a sign, so this is an elliptic operator.
\end{lem}

\begin{proof}
We follow the calculations in \cite{Mazzieri-Segatti} for the $\mathbb R^N\setminus\{0\}$ case. First
recall from \eqref{B-radial}, for $B_0(:=B_{g_{V_\ep}})$, that
\begin{equation}\label{B-radial2}\begin{split}
\big( B_{0}\big)_t^t & =    -\tfrac{n-2k}{4k} (V_\ep)^2  - V_\ep\ddot{V_\ep} +\tfrac{n-k}{n-2k}\dot{V_\ep}^2=:\kappa_1, \\
\big( B_{0}\big)_j^i & =  \left[\tfrac{n-2k}{4k}(V_\ep)^2-\tfrac{k}{n-2k} \dot{V_\ep}^2 \right] \delta_i^j =:\kappa_2\,\delta_i^j ,\\
\big( B_{0}\big)_\b^\a & =   -\tfrac{n-2k}{4k}\big(V_\ep+\tfrac{2k}{n-2k}\,\dot V_\ep\big)^2 \delta_\b^\a=:\kappa_3\,\delta_\b^\a,
\end{split}\end{equation}
and the rest of the entries of the matrix vanish. Moreover, for the diagonal matrix $B_0$,
\begin{equation}\label{S}\begin{split}
& [T^{k-1}(B_0)]_t^t= \sum_{m=0}^{k-1}(-1)^{k-1-m}\sigma_m(B_0)\kappa_1^{k-1-m}=:S_1,\\
&  [T^{k-1}(B_0)]_j^i= \left[\sum_{m=0}^{k-1}(-1)^{k-1-m}\sigma_m(B_0)\kappa_2^{k-1-m}\right]\delta_j^i=:S_2 \,\delta_i^j,\\
&  [T^{k-1}(B_0)]_\beta^\alpha= \left[\sum_{m=0}^{k-1}(-1)^{k-1-m}\sigma_m(B_0)\kappa_3^{k-1-m}\right]\delta_\beta^\alpha=:S_3\,\delta_\beta^\alpha.
\end{split}\end{equation}
Note that, since $g_{V_\ep}$ belongs to the positive  $\Gamma_k^+$ cone, then its $(k-1)$-Newton tensor $T^{k-1}$ is positive definite, so the quantities $S_1,S_2,S_3$ are strictly positive.
We recall from \eqref{Bs} that
\begin{equation*}\begin{split}
B_s&=\tfrac{n-2k}{2k}(V_\ep+sw)^2g_{cyl}^{-1}\,A_{g_{cyl}}-(V_\ep+sw)g_{cyl}^{-1}\,D^2(V_\ep+sw)\\
&+\tfrac{n}{n-2k}g_{cyl}^{-1}\,d(V_\ep+sw)\otimes d(V_\ep+sw)-\tfrac{k}{n-2k}|d(V_\ep+sw)|^2_{g_{cyl}}I,
\end{split}\end{equation*}
from where it is easy to calculate its variation:
\begin{equation*}
\begin{split}
\Big.\frac{d\big( B_{s}\big)_t^t}{ds}\Big|_{s=0} & =   -V_\ep \,\partial_{tt}w+\tfrac{2(n-k)}{n-2k}\, \dot V_
\ep\, (\partial_t w) - \left(\tfrac{n-2k}{2k} \,V_\ep+\ddot V_\ep\right)w, \\
\Big.\frac{d\big( B_{s}\big)_j^i}{ds}\Big|_{s=0} & =  -V_\ep \,g_\theta^{il}\,(D^2_\theta w)_{lj}-\tfrac{2k}{n-2k}\,\dot V_\ep \,\delta_j^i\,\partial_t w+\tfrac{n-2k}{2k}\,V_\ep\,\delta_j^i \,w,\\
\Big.\frac{d\big( B_{s}\big)_\b^\alpha}{ds}\Big|_{s=0} & =  -V_\ep\,e^{-2t} \partial_{\alpha\beta} w +\left[-\tfrac{2k}{n-2k} \,\dot V_\ep-V_\ep\right]\delta_\beta^\alpha\,\partial_t w\\&\quad+\left[-\tfrac{n-2k}{2k}\,V_\ep +\dot V_\ep\right]w\delta_\beta^\alpha.
\end{split}
\end{equation*}
Substituting \eqref{B-radial2} above one arrives at
\begin{equation*}\begin{split}
  \Big.\frac{d\big( B_{s}\big)_t^t}{ds}\Big|_{s=0} & =   (V_\ep)^{-1}\,(B_0)_t^t\,w -\tfrac{n-k}{n-2k}\,(V_\ep)^{-1}(\dot V_\ep)^2w-V_\ep\,\partial_{tt} w\\&\quad+\tfrac{2(n-k)}{n-2k}\dot V_\ep\partial_t w-\tfrac{n-2k}{4k}V_\ep w,\\
  \Big.\frac{d\big( B_{s}\big)_j^i}{ds}\Big|_{s=0} & =  (V_\ep)^{-1}\,w\, (B_0)_j^i+\tfrac{k}{n-2k}\,(V_{\ep})^{-1}(\dot V_\ep)^2 \,\delta_j^i\,w + \tfrac{n-2k}{4k}V_\ep \delta_j^i w\\&\quad-V_\ep (g_\theta)^{il} (D^2_\theta w)_{lj}
  -\tfrac{2k}{n-2k}\,\delta_j^i \,\dot V_\ep \,\partial_t w,   \\
  \Big.\frac{d\big( B_{s}\big)_\beta^\alpha}{ds}\Big|_{s=0} & =  (V_\ep)^{-1}\, w\,(B_0)_\beta^\alpha +\tfrac{k}{n-2k}\,(V_\ep)^{-1}(\dot V_\ep)^2 \delta_\beta^\alpha\,w-\tfrac{n-2k}{4k} V_\ep \delta_\beta^\alpha w \\&\quad+\left[-V_\ep-\tfrac{2k}{n-2k}\,\dot V_\ep\right]\delta_\beta^\alpha\,\partial_t w
  -V_\ep\, e^{-2t}\,\partial_{\alpha\beta} w.
  \end{split}
  \end{equation*}
Now, we can use the formula
$$k\sigma_k(B_0)=\trace \left(T^{k-1}(B_0)\cdot B_0\right),$$
to write
\begin{equation*}
\begin{split}
\Big.\frac{d}{ds}\Big|_{s=0}&\sigma_k(B_s)=k\sigma_k(B_0)\,(V_\ep)^{-1} w\\&+S_1\left[ -\tfrac{n-k}{n-2k}(V_\ep)^{-1} (\dot V_\ep)^2 w -\tfrac{n-2k}{4k}V_\ep w-V_\ep\,\partial_{tt}w+\tfrac{2(n-k)}{n-2k}\dot V_\ep\partial_t w\right]\\
&+S_2\left[\tfrac{k}{n-2k}(N-1) (V_\ep)^{-1}(\dot V_\ep)^2w+ (N-1)\tfrac{n-2k}{4k}V_\ep  w-V_\ep \Delta_\theta w\right.\\&\left.\qquad\quad-\tfrac{2k}{n-2k}(N-1)\dot V_\ep\partial_t w\right]\\
&+ S_3\left[\tfrac{k}{n-2k}p (V_\ep)^{-1} (\dot V_\ep)^2 w-p\tfrac{n-2k}{4k} V_\ep w+p\lp -V_\ep-\tfrac{2k}{n-2k}\dot V_\ep\rp\partial_t w\right.\\&\left.\qquad\quad-V_\ep e^{-2t}\Delta_z w\right].
\end{split}
\end{equation*}
Thus, using that $V_\ep$ is an exact solution to \eqref{equation-v} and formula \eqref{calculation-linearized2} for the linearization, we obtain that
\begin{equation*}
   L_\ep[w]=a^{\epsilon}_{0} w + a^{\epsilon}_{1} \partial_t w +a^{\epsilon}_{2}\partial_{tt} w+a^{\epsilon}_{3}\Delta_\theta w+a^{\epsilon}_{4} e^{-2t}\Delta_z w
\end{equation*}
for
\begin{equation}\label{a's}\begin{split}
&a^{\epsilon}_{0}=(k-q) c V_\ep^{q-1}-\tfrac{n-k}{n-2k}S_1V_\ep^{-1} \dot V_\ep^2 +\tfrac{k}{n-2k}(N-1)S_2V_\ep^{-1}\dot V_\ep^2\\&\quad+ \tfrac{k}{n-2k} p S_3 V_\ep^{-1}\dot V_\ep^2 -\tfrac{n-2k}{4k} S_1 V_\ep+(N-1)\tfrac{n-2k}{4k} S_2V_\ep\\&\quad-p\tfrac{n-2k}{4k}S_3V_\ep,\\
&a^{\epsilon}_{1}= \tfrac{2(n-k)}{n-2k}\dot V_\ep S_1-\tfrac{2k}{n-2k}(N-1)\dot V_\ep S_2+p\lp -V_\ep-\tfrac{2k}{n-2k}\dot V_\ep\rp S_3,\\
&a^{\epsilon}_{2}=-V_\ep S_1,\\
&a^{\epsilon}_{3}=-V_\ep S_2,\\
&a^{\epsilon}_{4}=- V_\ep S_3.
\end{split}\end{equation}
This completes the proof of the Lemma.
\end{proof}

From now on we assume $k=2$ and fix $0<p<\PP_2$.  For the calculation of the inditial roots, it will be more convenient to work with the renormalized operator 
\begin{equation}\label{L'}
L'_\ep:=V_\ep^{-3} L_\ep.
\end{equation}
 Observe that this does not change the behavior as $t\to+\infty$ since, from  \eqref{behavior-zero}, we have that $V_\ep \to v_\infty$, a constant.

\subsection{Indicial roots as $t\to \infty$}\label{subsection:indicial-roots0}

Here we study limiting asymptotics of $L'_\ep$ as  $t\to +\infty$ (that is, as $r\to 0$). Remark that  both $L_\ep$  and $L'_\ep$ have constant coefficients in the limit. Thus we can define
\begin{equation*}
b_\ell=v_\infty^{-3}\lim_{t\to +\infty} a^\ep_\ell(t),\quad \ell=0,1,2,3,4.
\end{equation*}
We will give a precise formula  for these coefficients below (in particular, they do not depend on $\ep$).

 Using \eqref{B-radial2} we can show that the matrix $B_0(t)$ converges to $B_\infty$ for
\begin{equation*}\begin{split}
(B_\infty)_t^t & =    -\tfrac{n-2k}{4k} (v_\infty)^2,  \\
(B_\infty)_j^i & =  \tfrac{n-2k}{4k}(v_\infty)^2\delta_i^j ,\\
 (B_\infty)_\b^\a & =   -\tfrac{n-2k}{4k}(v_\infty)^2 \delta_\b^\a.
\end{split}\end{equation*}
Then from  \eqref{S} we have
\begin{equation*}
\begin{split}
&S_1^\infty := [T^1(B_\infty)]_t^t=\tfrac{n-4}{8} v_\infty^2(n-2p-1),\\
&S_2^\infty := [T^1(B_\infty)]_j^i=\tfrac{n-4}{8} v_\infty^2(n-2p-3),\\
&S_3^\infty := [T^1(B_\infty)]_\beta^\alpha=\tfrac{n-4}{8} v_\infty^2(n-2p-1),\\
\end{split}
\end{equation*}
which yields
\begin{equation*}\begin{split}
&b_0=\left[(2-q) c (v_\infty)^{q-4}-\lp\tfrac{n-4}{8}\rp^2 (n-2p-1) \right.\\
&\qquad+\left.(N-1)\lp\tfrac{n-4}{8}\rp^2 (n-2p-3) -p\lp\tfrac{n-4}{8}\rp^2 (n-2p-1)\right],\\
&b_1=-p\tfrac{n-4}{8}(n-2p-1),\\
&b_2=-\tfrac{n-4}{8}(n-2p-1),\\
&b_3=-\tfrac{n-4}{8}(n-2p-3),\\
&b_4=-\tfrac{n-4}{8}(n-2p-1).
\end{split}\end{equation*}
Recalling \eqref{v-infty}, we simplify $b_0$ to
\begin{equation*}
\begin{split}
b_0&=\lp\tfrac{n-4}{8}\rp^2\left[(2-q) c_{n,p,2} - (p+1)(n-2p-1) +(N-1) (n-2p-3)\right]\\
&=-\lp\tfrac{n-4}{8}\rp(4p^2+8p-4np-5n+4+n^2).
\end{split}
\end{equation*}

The behavior of $L'_\ep$ when $t\to +\infty$ (or $r\to 0)$ is given by the normal operator
\begin{equation}\label{Linftyyy}
   L^{(0)}[w]=b_0 w + b_1\partial_t w +b_2\partial_{tt} w+b_3\Delta_\theta w+b_4 e^{-2t}\Delta_z w,
\end{equation}
and the indicial operator
\begin{equation}\label{mathcalLinfty}
   L^{(0)}_\natural [w]=b_0 w + b_1\partial_t w +b_2\partial_{tt} w+b_3\Delta_\theta w.
\end{equation}

We consider now the spherical harmonic decomposition of $\mathbb S^{N-1}$ and project the operators \eqref{Linftyyy} and \eqref{mathcalLinfty} over each eigenspace. For this, we set
\begin{equation*}
    L_{j}^{(0)}[w]=b_0 w + b_1\partial_t w +b_2\partial_{tt} w-b_3\lambda_j w+b_4 e^{-2t}\Delta_z w
\end{equation*}
and
\begin{equation*}
 L_{\natural, j}^{(0)}[w]=b_0 w + b_1\partial_t w +b_2\partial_{tt} w-b_3\lambda_j w.
\end{equation*}
We look for solutions of $ L_{\natural,j}^{(0)}[w]=0$ of the form $w(t)=e^{-\gamma t}$. Such $\gamma$ must satisfy the quadratic equation
\begin{equation*}
f_j(\gamma):=b_0-b_3\lambda_j -b_1 \gamma+b_2\gamma^2=0.
\end{equation*}
Observe that $b_1,b_2,b_3<0$.
An elementary analysis of these parabolas yields:

\begin{lemma}\label{lemma:indicial-origin} For each
 $j=0,1,\ldots$, there exist two indicial roots $\gamma_j^\pm$, given as solutions of $f_j(\gamma)=0$. These satisfy:
\begin{itemize}
\item $\gamma_0^\pm$ can be real or complex. If they are real and different, then
\begin{equation}\label{gamma0}0<\gamma_0^-<\tfrac{p}{2}<\gamma_0^+,\end{equation}
otherwise it holds
\begin{equation}\label{gamma00}\re \gamma_0^\pm=\tfrac{p}{2}>0.\end{equation}
\item For $j\geq 1$, $\gamma_{j}^\pm$ are real numbers. In addition, we have the monotonicity
\begin{equation}\label{indicial-zero2}
\ldots\leq\gamma_2^-\leq \gamma_1^-< \re\gamma_0^-\leq\tfrac{p}{2}\leq \re\gamma_0^+<\gamma_1^+\leq \gamma_2^+\leq\ldots
\end{equation}
\item It holds 
\begin{equation}\label{gamma1}\gamma_1^-=-1.
\end{equation}
\end{itemize}
\end{lemma}

\begin{proof}
The graph of the second order polynomial $f_j(\gamma)$ is given by a downward-open parabola, with vertex located at  $\gamma=\frac{b_1}{2b_2}=\frac{p}{2}$.

Let us look first at the case  $j=0$. First remark that $f_0(0)<0$. Second, we may or may not have real roots for $f_0(\gamma)=0$, however, in any case, either \eqref{gamma0} or \eqref{gamma00} holds  under the hypothesis $0<p<\PP_2$ (and this is sharp). 

To show \eqref{gamma1},  just notice that the polynomial $f_1(\gamma)$ has roots exactly at $\gamma=-1$ and $\gamma=p+1$,

Finally, to prove \eqref{indicial-zero2} one sees that  $f_j(\gamma)$ is non-decreasing in $j$, that is, the parabolas shift vertically as $j$ grows. Thus if the roots are real for $j=1$, they will be also for $j\geq 2$, and have the required monotonicity in $j$.

\end{proof}

Observe that it is possible to give a self-adjoint  version of $L_\ep$, denoted by $\widetilde L_\ep$, of the form
\begin{equation}\label{L-tilde-epsilon}
   \widetilde L_\ep[\widetilde{w}]=\tilde a_0^\ep \widetilde w + \tilde a_2^\ep\partial_{tt} \widetilde w+\tilde a_3^\ep\Delta_\theta \widetilde w+\tilde a_4^\ep e^{-2t}\Delta_z \widetilde w,
\end{equation}
for some coefficients $\tilde a_\ell^\ep$ which can be calculated from the original $a_\ell^\ep$.  We will not need its precise expression, only the limit operator as $t\to \infty$. More precisely, the \emph{conjugate} operator to $L^{(0)}$ is defined by
\begin{equation}\label{tilde-L0}
\widetilde L^{(0)}[\widetilde w]=e^{-\frac{p}{2} t}L^{(+\infty)}[ e^{\frac{p}{2} t}\widetilde w]=\tilde b_0 \widetilde w
+b_2\partial_{tt}\widetilde w+b_3\Delta_\theta \widetilde w+b_4e^{-2t}\Delta_z \widetilde w,
\end{equation}
where we have defined
\begin{equation*}
\tilde b_0=b_0-\frac{b_1^2}{4b_2}.
\end{equation*}

\begin{remark}\label{indicial-delta}
The advantage of the conjugate operator $\widetilde{L_\ep}$ is that it is self-adjoint and the indicial roots as $t\to +\infty$ (which will be  denoted by $\delta_j^{\pm})$ have the same structure as those in Lemma \ref{lemma:indicial-origin}, but are centered at the origin instead of $\frac{p}{2}$, fact that makes the Hilbert space analysis more clear. 

\end{remark}

\subsection{Indicial roots as $t\to -\infty$}\label{subsection:indicial-rootsinfty}

We now calculate the indicial roots of $L_\ep$ as   $t\to -\infty$ (that is, $r\to\infty$). Again, we work with the renormalized operator $L'_\ep=V_\ep^{-3}L_\ep$. We will see that the following limits exist:
\begin{equation*}
d_\ell=\lim_{t\to -\infty} V_\ep^{-3}(t)a^\ep_\ell(t),\quad \ell=0,1,2,3,4.
\end{equation*}
Since $V_\ep$ behaves as $\ep^{\alpha_0} e^{\alpha_0 t}$,  we have that $B_0(t)$ converges to $B_{-\infty}$ as $t\to -\infty$, where
\begin{equation*}\begin{split}
&(B_{-\infty})_t^t  =V_\ep^2 \left( -\tfrac{n-2k}{4k}+\tfrac{k}{n-2k}\alpha_0^2\right),  \\
&(B_{-\infty})_j^i  =V_\ep^2\left(\tfrac{n-2k}{4k}-\tfrac{k}{n-2k}\alpha_0^2\right)\delta_i^j ,\\
& (B_{-\infty})_\b^\a  =V_\ep^2   \left(-\tfrac{n-2k}{4k}\right)\left(1+\tfrac{2k}{n-2k}\alpha_0\right)^2\delta_\b^\a,
\end{split}\end{equation*}
which implies
\begin{equation*}
\begin{split}
&S_1^{-\infty}= V_\ep^2\left(\tfrac{n-4}{8}(n-2p-1)-\alpha_0p+\alpha_0^2 2\tfrac{-n+1}{n-4}\right)=:V_\ep^2s_1,\\
&S_2^{-\infty}=V_\ep^2\left(\tfrac{n-4}{8}(n-2p-3)-\alpha_0p+\alpha_0^2 2\tfrac{-n+3}{n-4}\right)=:V_\ep^2s_2,\\
&S_3^{-\infty}=V_\ep^2\left(\tfrac{n-4}{8}(n-2p-1)-\alpha_0(p-1)+\alpha_0^2 2\tfrac{-n+3}{n-4}\right)=:V_\ep^2s_3.
\end{split}
\end{equation*}
Then, the expressions in \eqref{a's} yield
\begin{equation}\label{a'ss}\begin{split}
&d_0=\left\{-\tfrac{n-2}{n-4}s_1\alpha_0^2 +\tfrac{2}{n-4}(n-p-1)s_2\alpha_0^2\right.\\&\left.\quad+ \tfrac{2}{n-4} p s_3 \alpha_0^2 -\tfrac{n-4}{8} s_1 +(n-p-1)\tfrac{n-4}{8} s_2-p\tfrac{n-4}{8}s_3\right\},\\
&d_1= \left\{\tfrac{2(n-2)}{n-4}\alpha_0  s_1-\tfrac{4}{n-4}(n-p-1)\alpha_0 s_2\right.\\&\quad+\left.p\lp -1-\tfrac{4}{n-4}\alpha_0\rp s_3\right\},\\
&d_2=-\left(\tfrac{n-4}{8}(n-2p-1)-\alpha_0p+\alpha_0^2 2\tfrac{-n+1}{n-4}\right),\\
&d_3=-\left(\tfrac{n-4}{8}(n-2p-3)-\alpha_0p+\alpha_0^2 2\tfrac{-n+3}{n-4}\right),\\
&d_4=- \left(\tfrac{n-4}{8}(n-2p-1)-\alpha_0(p-1)+\alpha_0^2 2\tfrac{-n+3}{n-4}\right).
\end{split}\end{equation}
Indicial roots are calculated as the roots of the quadratic polynomial
\begin{equation*}
d_0-d_3\lambda_j -d_1 \vartheta+d_2\vartheta^2=0, \quad j=0,1,\ldots
\end{equation*}
Even though we could proceed as in Subsection \ref{subsection:indicial-roots0} to analyse these parabolas, the expressions of the coefficients from \eqref{a'ss} are so involved that we prefer to take a different path, and we refer the reader to Section \ref{subsction:model-linearization} below. In any case, for each $j=0,1,\ldots$, there exist two indicial roots $\vartheta_j^\pm$, with similar monotonicity.

\section{The linearized operator}\label{section:linearized}

Now we go back to the general setting $M\setminus \Lambda$ with the metric given in \eqref{background-metric-u}, where  $\bar u_\ep$ is the approximate solution from Section \ref{subsection:approximate-sol} above. We let  $\mathbb L_\ep:=\mathbb L(\bar u_\ep,g_{\mathbb R^n})$ be the linearized operator around $\bar u_\ep$,  as defined in \eqref{linope}, and take a conformal perturbation of the metric $\bar g_\ep$, i.e, we set for $s\in\mathbb R$,
$$s\mapsto g_s:=(\bar u_\ep+s\varphi)^{\frac{8}{n-4}}g_{M}.$$
Let $\mathbb B_s$ be the symmetric $(1,1)$-tensor given by
\begin{equation}
\label{Bs0}\mathbb B_s:=\tfrac{n-4}{4} (\bar u_\ep+s\varphi)^{\frac{2n}{n-4}} g_s^{-1} A_{g_s}.
\end{equation}
We have
\begin{equation}\label{linearized-general}
\mathbb L_\ep[\varphi]:=\left.\frac{d}{ds}\right|_{s=0}\sigma_2(\mathbb B_s)-cq\bar u_\ep^{q-1}\varphi.
\end{equation}

\subsection{The model linearization}\label{subsction:model-linearization}

We first consider $\mathcal L_\ep$, the linearized operator around the model solution $U_\ep$ in $\mathbb R^n\setminus \mathbb R^p$, using the notation from Section \ref{subsection:cyl-coordinates}. We will translate the results from Section \ref{section:explicit} regarding operator $L_\ep$ in cylindrical coordinates back to the variable $r$. For this, recall the relation \eqref{rewrite}:

Then the arguments of the previous section immediately yield:

\begin{pro}\label{prop:expression}
The operator $\mathcal L_\ep$ has the  following expression:
\begin{equation}\label{mathcalL}
\mathcal L_\ep[\varphi]=\mathcal A_0^\ep(r) \varphi +  \frac{\mathcal A_1^\ep(r)}{r}\partial_r  \varphi +\mathcal A_2^\ep(r)\partial_{rr} \varphi+\frac{\mathcal A_3^\ep(r)}{r^2}\Delta_\theta \varphi+\mathcal A_4^\ep(r) \Delta_z \varphi,
\end{equation}
 for some coefficients $\mathcal A^\ep_\ell(r)$, $\ell=0,1,2,3,4$,  satisfying
$$\mathcal A^\ep_2,\mathcal A^\ep_3,\mathcal A^\ep_4>0.$$
\end{pro}

\begin{remark}
These coefficients are smooth except at $r=0$ and $r=\infty$. Precise asymptotics at those points  will be given below. 
\end{remark}


Now we take the limits as $r\to 0$ and $r\to \infty$. As we showed in the previous subsection, it is more convenient to work with a normalized linearization \eqref{L'}, denoted with a `prime'. Following this convention, we set
\begin{equation}\label{L-prima}
\mathcal L'_\ep=U_\ep^{-3}\mathcal L_\ep.
\end{equation}

From the calculations in Subsection \ref{subsection:indicial-roots0}, as $r\to 0$ the operator $\mathcal L_\ep'=r^{\frac{n-4}{2}+\frac{n}{4}}\mathcal L_\ep$ converges to the constant coefficient operator
\begin{equation}\label{L0}
\mathcal L^{(0)}[\varphi]:=\beta_0^{(0)} \varphi +  \frac{\beta_1^{(0)}}{r}\partial_r  \varphi +\beta_2^{(0)}\partial_{rr} \varphi+\frac{\beta_3^{(0)}}{r^2}\Delta_\theta \varphi+\beta_4^{(0)} \Delta_z \varphi,
\end{equation}
for some constants $\beta_\ell^{(0)}$, $\ell=0,1,2,3,4$.

Moreover, in Subsection \ref{subsection:indicial-rootsinfty} we checked that, as $r\to\infty$, the operator $\mathcal L'_\ep$ converges to
\begin{equation}\label{Linfty}
\mathcal L^{(\infty)}[\varphi]:=\beta_0^{(\infty)} \varphi +  \frac{\beta_1^{(\infty)}}{r}\partial_r  \varphi +\beta_2^{(\infty)}\partial_{rr} \varphi+\frac{\beta_3^{(\infty)}}{r^2}\Delta_\theta \varphi+\beta_4^{(\infty)} \Delta_z \varphi,
\end{equation}
for some constant coefficients  $\beta^{(\infty)}_\ell$, $\ell=0,1,2,3,4$.

The fact that the limit operators $\mathcal L^{(0)}$ and $\mathcal L^{(\infty)}$ have constant coefficients (independent of $\ep$) allow to  characterize all the indicial roots for $\mathcal L_\ep$, as follows:

\begin{pro}\label{prop:indicial}
For each $j=0,1,\ldots$, there exist two indicial roots $\chi_{j,\pm}^{(0)}$ for the operator $\mathcal L_\ep$ as $r\to 0$. These satisfy:
\begin{itemize}
\item $\chi_{0,\pm}^{(0)}$ can be real or complex. In the former case,
$$-\tfrac{n-4}{4}<\chi^{(0)}_{0,-}<\tfrac{p}{2}-\tfrac{n-4}{4}<\chi^{(0)}_{0,+},$$
 while in the latter, $\re (\chi^{(0)}_{0,\pm})=\frac{p}{2}-\frac{n-4}{4}$.
\item For $j\geq 1$, $\chi_{j,\pm}^{(0)}$ are real numbers. In addition, we have the monotonicity
\begin{equation}\label{indicial-zero}
\begin{split}
\ldots\leq\chi_{2,-}^{(0)}\leq \chi_{1,-}^{(0)}< \re(\chi_{0,-}^{(0)})&\leq\tfrac{p}{2}-\tfrac{n-4}{4}\\
&\leq \re(\chi_{0,+}^{(0)})<\chi_{1,+}^{(0)}\leq \chi_{2,+}^{(0)}\leq\ldots
\end{split}
\end{equation}
\item It holds $\chi_{1,-}^{(0)}=-1-\frac{n-4}{4}$.
\end{itemize}

As $r\to\infty$, the picture is similar and, for each $j=0,1,\ldots$, there exist two indicial roots $\chi_{j,\pm}^{(\infty)}$ for  $\mathcal L_\ep$ which satisfy:
\begin{itemize}
\item For all $j$, $\chi_{j,\pm}^{(\infty)}$ are real numbers. In addition, we have the monotonicity
\begin{equation}\label{indicial-infinity0}
\begin{split}
\ldots\leq\chi_{2,-}^{(\infty)}\leq \chi_{1,-}^{(\infty)}< \chi_{0,-}^{(\infty)}&\leq-\tfrac{n-4}{4}+\tfrac{p(n-3)}{2(n-1)}\\
&\leq \chi_{0,+}^{(\infty)}<\chi_{1,+}^{(\infty)}\leq \chi_{2,+}^{(\infty)}\leq\ldots
\end{split}
\end{equation}
\item It holds
\begin{equation}\label{indicial-infinity-first}
-\tfrac{n-4}{4}+\tfrac{p}{2}>\chi_{0,-}^{(\infty)}=-\tfrac{n-4}{4}+\tfrac{p(n-3)}{2(n-1)}-\tfrac{\sqrt{4p+5p^2-5pn+pn^2-p^2n}}{2(n-1)}.
\end{equation}
\item For $j=1$,
\begin{equation}\label{chi1-}
\chi_{1,-}^{(\infty)}=-1-\alpha_0-\tfrac{n-4}{4}
\end{equation}
and
\begin{equation}
\label{chi1+}\chi_{1,+}^{(\infty)}>0.
\end{equation}
\end{itemize}
\end{pro}

\begin{proof}
The indicial roots at the origin were calculated in  Subsection \ref{subsection:indicial-roots0}, taking into account that the relation \eqref{rewrite} introduces a shift of $\frac{n-4}{4}$ in the indicial roots between $L_\ep$ and $\mathcal L_\ep$.

We started studying the behavior at infinity  in  Subsection \ref{subsection:indicial-rootsinfty}. Full details of the proof will be  postponed until Section \ref{section:injectivity}. However, let us give here the main ideas:

We start with  the $j=0$ mode, that corresponds to radially symmetric solutions. The fundamental observation is that knowing  an exact solution for the non-linear ODE \eqref{ODE-sigma2}, one can produce radially symmetric solutions of its linearized equation using dilation invariance, as explained in  Remark \ref{remark:dilation-inv}. Thus we know precisely the behavior at infinity of a particular solution $\varphi_\sharp$, which agrees with the power $r^{-\alpha_0^\pm}$. The precise value of $\alpha_0^\pm$ was calculated in the paper \cite{Gonzalez-Mazzieri} and it is given by
\begin{equation}\label{alpha0pm}
\alpha_0^{\pm}=\frac{n-4}{4}-\frac{p(n-3)}{2(n-1)}\pm\frac{\sqrt{4p+5p^2-5pn+pn^2-p^2n}}{2(n-1)},
\end{equation}
(we denote $\alpha_0=\alpha_0^-$). From here we conclude that  $-\alpha_0^{\pm}$ are the two indicial roots for $j=0$.

Note also  that the expression inside the square root in formula \eqref{alpha0pm} is always positive in our range of $p$. This implies that, for all $j$, all the indicial roots  are real numbers.  Summarizing this discussion, we have shown \eqref{indicial-infinity0}.

We can also give explicit formulas for $j=1$. Using rotational invariance from Remark \ref{remark:rotation-invariance} one can find an explicit solution in the kernel for $j=1$, and this yields \eqref{chi1-}. Finally, to show \eqref{chi1+}, one just needs to check that
\begin{equation*}
-\tfrac{n-4}{4}+p\tfrac{n-3}{n-1}+1+\alpha_0>0,
\end{equation*}
which follows from simple algebra for $n\geq 5$.
\end{proof}

Next, we give a ``divergence-form" version for $\mathcal L_\ep$ with the introduction of an  integrating factor. Indeed, let
\begin{equation*}
\mathcal H^\ep_1(r)=\exp\int^r \frac{\mathcal A_1^\ep(s)}{s\mathcal A_2^\ep(s)}\,ds \quad\text{and}\quad (\mathcal H^\ep)^{-1}(r)=\frac{\mathcal A_2^\ep(r)}{\mathcal H^\ep_1(r)}.
\end{equation*}
Note that both functions are strictly positive. It holds
\begin{equation}\label{HH}
\mathcal L_{\ep}\varphi=(\mathcal H^\ep)^{-1}(r)\partial_r\left\{\mathcal H^\ep_1(r)\partial_r \varphi\right\}+\mathcal A_0^\ep(r) \varphi +\frac{\mathcal A_3^\ep(r)}{r^2}\Delta_\theta\varphi+\mathcal A_4^\ep(r) \Delta_z \varphi,
\end{equation}
which shows that a natural space to work is $L^2(\mathbb R^n\setminus\mathbb R^p)$ with the (weighted) scalar product
\begin{equation}\label{space-H}
\langle \varphi_1,\varphi_2  \rangle=\int_0^\infty \int_{\mathbb S^{N-1}}\int_{\mathbb R^p}\mathcal H^\ep(r) \varphi_1\varphi_2 \,dz d\theta dr,
\end{equation}
for which $\mathcal L_\ep$ is self-adjoint. 

A possibility to avoid using this integrating factor is to work with the conjugated operator $\widetilde L_\ep$ from \eqref{L-tilde-epsilon}, and this will be our approach for the $L^2$ theory.

\subsection{A closer look at $\mathbb L_\ep$}\label{subsection:closer-look}

Let us go back to the general case of a general singular set $\Lambda$ and consider  $\mathbb L_\ep=\mathbb L(\bar u_\ep,g_M)$,  the linearized operator around the approximate solution $\bar u_\ep$.   

Motivated by \eqref{L-prima}, it is preferable to use the normalization
\begin{equation}\label{normalization-L}
\mathbb L'_\ep=\bar u_\ep^{-3}r^{2}\mathbb L_\ep,
\end{equation}
where $r$ is extended to one away from the singular set.

In a tubular neighborhood  $\mathcal T_\varrho(\Lambda)$ around the singular set, the background metric in Fermi coordinates is written as \eqref{metric-Fermi}, which can be compared to the model \eqref{metric-gE}. Thus we can prove:

\begin{pro}\label{prop:compare-operators}
For functions supported in $\mathcal T_\varrho(\Lambda)$ that do not depend on the $z$ variable, we have that
\begin{equation}\label{compare-operators}\mathbb L'_\ep=\mathcal L'_\ep+\mathcal D,
\end{equation}
where $\mathcal D$ is a second order differential operator (at most) that satisfies, for functions of the form $\varphi=O(r^a)$ as $r\to 0$,
\begin{equation*}
\mathcal D \varphi=O(r^{a-2+\beta}),
\end{equation*}
for some $\beta>0$.
\end{pro}

\begin{remark}\label{remark:dependence-z}
To handle the dependence on the variable $z$, we just need to recall that our arguments rely on localization and rescaling near a fixed point $\Lambda$, around which we use normal coordinates.
\end{remark}

\begin{proof}[Proof of Proposition \ref{prop:compare-operators}]
As in \cite[Section 4.2]{Mazzeo-Pacard}, in the neighborhood $\mathcal T_\rho(\Lambda)$, the difference between
$\mathbb L'_\ep$  and $\mathcal L'_\ep$ comes from the extra terms of the metric in $\mathbb R^n\setminus \Lambda$ from \eqref{metric-Fermi} with respect to  the model metric in $\mathbb R^n\setminus\mathbb R^p$ from \eqref{metric-gE} as we did in the proof of Proposition \ref{prop-gluing-Lambda}. Controlling the  error terms in the metric, which are of order $o(1)$  as $r\to 0$, yields the desired result.
\end{proof}

Now, away from the singular set, say in $M\setminus\mathcal T_\varrho$, the normalization in $\mathbb L'_\ep$ cancels the rescaling of the metric \eqref{rescaling-metric} by $\ep^{\alpha_0}$, so that we recover $\mathbb L'_\ep=\mathbb L_\ep(u_*,g_M)$ which is uniformly elliptic and non-degenerate, with constants not depending on   $\epsilon$. We summarize this discussion in the remark below.

\begin{remark}\label{remark:L0} The normalization $\mathbb L'_\ep$ behaves as:
\begin{itemize}
\item Away from the singularity, we have a usual uniformly elliptic, non-degenerate operator, independently of $\epsilon$.

\item Near the singularity set, $\mathbb L'_\ep$ is in the class of elliptic edge operators from \cite{Mazzeo:edge}, and it is modelled after $\mathcal L'_\ep$. In particular, it has constants indicial roots as $r\to 0$ (independent of $\ep$).
\end{itemize}
\end{remark}

In order to simplify the functional analysis for $\mathbb L_\ep$ (or $\mathbb L'_\ep$), a possibility  is to work with the ``self-adjoint" version of $\mathbb L_\ep$, which near the singular set is essentially \eqref{HH} with the scalar product \eqref{space-H}. 

An equivalent approach is to follow the calculations in \eqref{L-tilde-epsilon}, passing to the Fermi variable $t=-\log r$ and introducing the  ``conjugate'' operator $\widetilde{\mathbb L}_\ep$ (and the corresponding $\widetilde{\mathbb L}'_\ep$). This conjugation preserves $\mathbb L_\ep$ away from the singular set but, near $\Lambda$, reduces to a perturbation of  \eqref{tilde-L0}.  This makes the analytic setup simpler (see Remark \ref{indicial-delta}).

Finally, note that  $\widetilde{\mathbb L}_\ep$ is not self-adjoint in general, nevertheless, motivated by the above,  we will assume that it is self-adjoint, since it does not change the analysis below and avoids the use of extra notation. Same comment for $\widetilde{\mathbb L}'_\ep$.

\section{Function spaces}\label{section:function-spaces}

The objective of this section is to set up the functional analytic framework on $M\setminus\Lambda$. As we have discussed, $\mathbb L_\ep$ is  second order linear, elliptic operator, uniformly elliptic away from the singular set $\Lambda$ where it has the structure of an edge operator from \cite{Mazzeo:edge}.

\subsection{Weighted H\"older spaces}

We now define the weighted H\"{o}lder spaces $\mathcal C_{\mu}^{2,\alpha}(M\setminus\Lambda)$, following the notations and definitions from Section 3 in \cite{Mazzeo-Pacard}. Intuitively, these spaces consist of functions which are products of powers of the distance to $\Lambda$ with functions whose H\"{o}lder norms are invariant under homothetic transformations centered at an arbitrary point on $\Lambda$.

Let $u$ be a function in a tubular neighborhood $\mathcal T:=\mathcal T_\rho$ of $\Lambda$ and define
\begin{equation*}
\|u\|_{0,\alpha,0}^{\mathcal{T}}=\sup_{y\in \mathcal{T}}|u|+\sup_{y,y'\in \mathcal{T}}\frac{(r+r')^\alpha|u(y)-u(y')|}
{|r-r'|^\alpha+|z-z'|^\alpha+(r+r')^\alpha|\theta-\theta'|^\alpha},
\end{equation*}
where $y,y'$ are two points in $\mathcal{T}$ and $(r,\theta,z), (r',\theta', z')$ their Fermi coordinates.

\begin{defi}\label{def:holder-spaces}
The space $\mathcal C_0^{l,\alpha}(M \setminus \Lambda)$ is defined to be the set of all $u\in \mathcal C^{l,\alpha}(M\setminus \Lambda)$ for which the norm
\begin{equation*}
\|u\|_{\mathcal C^{l,\alpha}_0}=\|u\|_{\mathcal C^{l,\alpha}(M\setminus \mathcal T_{\rho/2})}+\sum_{j=0}^l\|\nabla^j u\|_{0,\alpha,0}^{\mathcal{T}}
\end{equation*}
is finite.

Now we consider a function $d_\Lambda$  behaving as the Fermi coordinate $r$ in a tubular neighborhood $\mathcal T_\rho$ of $\Lambda$ and a positive constant (say, identically one) away from the singular set. Then a function $u$ belongs to $\mathcal C_\mu^{l,\alpha}(M\setminus\Lambda)$ if and only if
\begin{equation*}
u=(d_\Lambda)^\mu \hat u \quad\text{for some}\quad \hat u\in \mathcal C^{l,\alpha}(M\setminus\Lambda).
\end{equation*}
This space is endowed with the natural norm
\begin{equation*}
\|u\|_{\mathcal C^{l,\alpha}_\mu}:=\|(d_\Lambda)^{-\mu} u\|_{\mathcal C^{l,\alpha}_0}.
\end{equation*}
\end{defi}
\noindent Basic properties of these norms can be found in \cite[Section 3]{Mazzeo-Pacard}.\\

In the particular setting of $\mathbb R^n\setminus \mathbb R^p$ it will be necessary to introduce  weighted H\"older spaces with respect to the $r$ variable for functions having different behaviors near $r=0$ and $r=\infty$.

First, in the case of an isolated singularity, this is, $\mathbb R^N\setminus\{0\}$, given any $\mu_1,\mu_2\in \R$, for $R>0$ fixed we set
\begin{equation*}\begin{split}
\mathcal C_{\mu_1}^{l,\alpha}(B_R \setminus \{0\})&=\{u=r^{\mu_1} \phi\,:\, \phi \in \mathcal C_0^{l,\alpha}(B_R \setminus \{0\})\},\\
\mathcal C_{\mu_2}^{l,\alpha}(\R^N \setminus B_R)&=\{u=r^{\mu_2} \phi\,:\, \phi\in \mathcal C_0^{l,\alpha}(\R^N \setminus B_R )\},
\end{split}
\end{equation*}
and thus we can define:

\begin{defi}
The space $\mathcal C^{l,\alpha}_{\mu_1,\mu_2}(\R^N \setminus \{0\})$ consists of all functions $u$ for which the norm
\begin{equation*}
\|u\|_{\mathcal C^{l,\alpha}_{\mu_1,\mu_2}}=\sup_{B_R\setminus \{0\}}\|r^{-\mu_1}u\|_{l,\alpha,0}+\sup_{\R^N\setminus B_R}\|r^{-\mu_2}u\|_{l,\alpha,0}
\end{equation*}
is finite. The spaces  $\mathcal C^{l,\alpha}_{\mu_1,\mu_2}(\R^n \setminus \R^p)$ are defined similarly, in terms of the (global) Fermi coordinates  $(r,\theta,z)$ and the weights $r^{\mu_1}$, $r^{\mu_2}$.
\end{defi}

\subsection{Weighted Sobolev spaces}\label{subsection:Sobolev}

We define, for $\delta\in\mathbb R$, the norm
\begin{equation}\label{norm-L2delta}
\|\varphi\|^2_{L^2_\delta(M\setminus\Lambda)}=\int_{M\setminus\mathcal T_\rho} \varphi^2\,dvol +
\int_0^{\rho}\int_{\mathbb S^{N-1}}\int_\Lambda \varphi^2 r^{\frac{n-4}{2}-p-2\delta-1}\,dzd\theta dr.
\end{equation}
The last term in the expression \eqref{norm-L2delta} above has a more user friendly expression in the variable $t=-\log r$. Indeed, using the same notation as in \eqref{tilde-L0}, if we set 
\begin{equation}\label{change:phiw}
\varphi = (d_\Lambda)^{-\frac{n-4}{4}}w=(d_\Lambda)^{-\frac{n-4}{4}+\frac{p}{2}}\widetilde w
\end{equation}
then, taking into account that $d_\Lambda\equiv r$ near the singular set,  it simplifies to
\begin{equation*}
\int_{-\log \rho}^{+\infty}\int_{\mathbb S^{N-1}}\int_\Lambda \widetilde w^2 e^{2\delta t}\,dzd\theta dt.
\end{equation*}

Finally, weighted Sobolev spaces $W^{k,2}_\delta$ are defined similarly.


\subsection{Duality}

We will consider the spaces $L^2_\delta(M\setminus\Lambda)$ and $L^2_{-\delta}(M\setminus\Lambda)$
to be dual with respect to the natural pairing
\begin{equation}\label{pairing}
L^2_\delta\times L^2_{-\delta}\ni (\widetilde w_1,\widetilde w_2)\mapsto \int_{M\setminus\Lambda}\widetilde w_1 \widetilde w_2.
\end{equation}
Thus, it is convenient to work in cylindrical variables and the operator $\widetilde{\mathbb L}_\ep$ in the $\widetilde w$ notation.

Let us look now at the model operator from \eqref{tilde-L0}, which already contains the structure near the singular set  of $\widetilde{\mathbb L}_\ep$. Fixed $\delta\in\mathbb R$, the dual of
\begin{equation*}
\widetilde{L}^{(0)} : L^2_{\delta}(\mathbb R^n\setminus \mathbb R^p)\to L^2_{\delta}(\mathbb R^n\setminus \mathbb R^p)
\end{equation*}
is given by
\begin{equation}\label{relation-adjoint}
(\widetilde{L}^{(0)})^* = r^{-2\delta} \widetilde{L}^{(0)} r^{2\delta}:L^2_{-\delta}(\mathbb R^n\setminus \mathbb R^p)\to L^2_{-\delta}(\mathbb R^n\setminus \mathbb R^p).
\end{equation}
A similar duality holds for  $\widetilde{\mathbb L}_\ep$ (or $\widetilde{\mathbb L}'_\ep$), as we have assumed it is self-adjoint. Relative to the pairing \eqref{pairing}, the adjoint of
$$\widetilde{ \mathbb  L}'_\ep:L^2_{\delta}(M\setminus\Lambda)\to L^2_{\delta}(M\setminus\Lambda)$$
 is 
$$(\widetilde{  \mathbb L}'_\ep)^*:L^2_{-\delta}(M\setminus \Lambda)\to L^2_{-\delta}(M\setminus \Lambda).$$

We will show in Proposition \ref{prop-not-so-easy} that  $\widetilde{\mathbb L}'_\ep$ (and thus, the original $\mathbb L_\ep$) is semi-Fredholm when $\delta>0$ not an indicial root. This implies that
\begin{equation}\label{ker}
\ker((\widetilde{  \mathbb L}'_\ep)^*)^\bot=\rango(\widetilde{  \mathbb L}'_\ep).
\end{equation}
Thus an easy way to prove that such $\widetilde{  \mathbb L}'_\ep$ is surjective is to check that its adjoint is injective.\\

\section{A priori estimates and $L^2$ semi-Fredholm properties}\label{section:Fredholm}

 Let $\mathbb L_\ep$ be the linearized operator around $\bar u_\ep$ in $M\setminus\Lambda$. Fredholm properties for this type of operators were shown in Mazzeo \cite{Mazzeo:edge} in great generality (using the theory of edge operators) and we refer to this paper for the complete proofs. Here, instead,  we consider a simpler PDE approach for the $L^2$ theory which was presented in the lecture notes
\cite{pacard}.


As explained above, it is more convenient to work in cylindrical coordinates and consider the conjugated  operator
 $\widetilde{\mathbb L}'_\ep: L^2_{\delta}(M\setminus\Lambda)\to L^2_{\delta}(M\setminus\Lambda)$.

Fixed $\ep>0$, $\widetilde{\mathbb L}'_\ep$ is linear and unbounded, densely defined and has closed graph. Our main result in this section proves (semi)-Fredholm properties, encoded in the a-priori estimate from Proposition \ref{prop-not-so-easy} for solutions of the equation
\begin{equation}\label{eq20}
\widetilde{  \mathbb L}'_\ep\widetilde w=\widetilde h \quad\text{in}\quad M\setminus\Lambda.
\end{equation}

Using the notation in  Remark \ref{indicial-delta}, we will denote by $\{\delta_j^\pm\}_j$ the indicial roots of $\widetilde{\mathbb L}'_\ep$ as $r\to 0$.

\begin{pro}\label{prop-not-so-easy}
Let $\delta\neq \delta_j^{\pm}$, $\delta>0$, and
take $\widetilde w\in L^2_\delta$, $\widetilde h\in L^2_{\delta}$ satisfying \eqref{eq20}. Then
\be\label{not-so-easy} \norm{\widetilde w}_{L^2_\delta(M\setminus\Lambda)}\leq C\lp \|
\widetilde h\|_{L^2_{\delta}(M\setminus\Lambda)} +\|\widetilde w\|_{L^2(\mathcal V)} \rp, \ee for  $\mathcal V$ any compact set in $M\setminus\Lambda$, and some constant $C(\mathcal V)$ not depending
on $\widetilde w$.
\end{pro}

Note that there is no dependence on $\epsilon$ in the conclusion of the Proposition.  The proof follows from a series of Lemmas:

\begin{lemma}(Localization)\label{lemma:localization}
It is sufficient to prove the Proposition for functions in $L^2_\delta$ supported in $\mathcal T_{\rho}(\Lambda)$ for some small $\rho$.
\end{lemma}

\begin{proof}
Introduce a cutoff $\chi$ identically one on  $\mathcal T_{\rho/2}(\Lambda)$, vanishing outside $\mathcal T_{\rho}(\Lambda)$. Then
\begin{equation*}
\widetilde h_1:=  \widetilde {\mathbb L}'_\ep (\widetilde w\chi)=\chi   \widetilde {\mathbb L}'_\ep \widetilde w+[ \widetilde{\mathbb L}'_\ep,\chi]\widetilde w.
\end{equation*}
Thus if  inequality \eqref{not-so-easy} is true for $\widetilde w\chi$, it is also true for $\widetilde w$ by adding a compactly supported term.
\end{proof}

\begin{lemma}(Reduction to the model case) For functions supported on $\mT_{\rho}(\Lambda)$ for some small $\rho$,
it is enough to prove \eqref{not-so-easy} for the model operator $\widetilde {L}^{(0)}$ instead of $\widetilde{  \mathbb L}'_\ep$.
\end{lemma}

\begin{proof}
First recall Proposition \ref{prop:compare-operators} (and Remark \ref{remark:dependence-z}) to reduce the problem to study the model operator ${L}'_\ep$ (or its conjugate $\widetilde {L}'_\ep$). Next, because of the ODE study from Proposition \ref{ODE-study}, we have for some $\varsigma>0$, in a small enough neighborhood $\{r<m\epsilon\}$ for some $m>0$,
$$\norm {D^\ell \lp v_\epsilon - v_\infty \rp}_{L^{\infty}(\mT_{\rho}(\Lambda))}\leq C_\ell \, e^{-\varsigma t},\quad \ell=0,1,\ldots$$
 Thus it is clear that
$$\widetilde {L}'_\ep=\widetilde {L}^{(0)}(1+O(e^{-\varsigma' t})),$$
for some $\varsigma'>0$.

Now, for points a bit further away from $\Lambda$, say, for $r>m\epsilon$, the operator $\widetilde L'_\ep$ does not depend on $\epsilon$, and it is a regular uniformly elliptic operator so standard Sobolev estimates hold.
\end{proof}

Now we give the proof of Proposition \ref{prop-not-so-easy} for the model $\widetilde{L}^{(0)}$. Assume, for now, that
$\delta_j:=\delta^{+}_j>0$ for all $j$. Recall the definition of the (conjugate) operator from \eqref{tilde-L0}; after projection over spherical harmonics, it becomes
$$\widetilde{ L}^{(0)}_{j}\widetilde w=(\tilde b_0-b_3\lambda_j)\widetilde w+b_2 \partial_{tt}\widetilde w +b_4 e^{-2t}\Delta_z \widetilde w,\quad b_2,b_4<0.$$
Take the Fourier transform in the variable $z$, and set
$$K_j \omega:= \tilde b_{0,j}\omega +b_2 \partial_{tt} \omega-b_4 e^{-2 t}|\zeta|^2\omega, $$
where we have defined $\tilde b_{0,j}:=\tilde b_0-b_3\lambda_j$, $\zeta$ is the Fourier variable and $\omega$ the Fourier transform of $\widetilde w$.
Now we make the change
$b_4|\zeta|^2 e^{-2 t}=b_2e^{-2\tau}$ and work in the variable $\tau$. This operator reduces  (up to a negative multiplicative constant) to
\begin{equation}\label{Kj}
K_j \omega=\partial_{\tau\tau}\omega-\delta_j^2\omega-e^{-2\tau}\omega,\quad \text{for }\omega=\omega(\tau).
\end{equation}
Define the space $L^2_\delta(d\tau)$ the one-dimensional weighted space with respect to the variable $\tau$ and weight $e^{2\delta \tau}$, and let us study the mapping properties of $K_j$ in $L^2_\delta(d\tau)$.

Note that for functions supported on $\tau\in(\tau_0,\infty)$ for $\tau_0$ big enough, the term  $-e^{-2\tau}$ is just a perturbation and can be ignored. Thus, for each fixed $j$, we consider the equation
\begin{equation}\label{problem-K}
\mathfrak K_j\omega=\psi\quad \mbox{for}\quad \mathfrak K_j\omega:=\partial_{\tau\tau}\omega-\delta_j^2\omega.
\end{equation}
Without loss of generality, we take $\tau_0=0$ in the next Lemma. The dependence on $\tau_0$ will be retaken in Lemma \ref{lemma:tau0}, in order be able to go from the variable $\tau$ back to the variable $t$.

Such  $\mathfrak K_j$   is a totally characteristic operator and has good Fredholm properties. Indeed:

\begin{lem}\label{lemma:not-so-easy-preliminary}
If $\delta$ not an indicial root, for every solution $\omega(\tau)$ of \eqref{problem-K} supported on $(0,\infty)$ we have
\begin{equation}\label{not-so-easy-preliminary}
\|\omega\|_{L^2_\delta(d\tau)}\leq C\|\psi\|_{L^2_\delta(d\tau)}.
\end{equation}
\end{lem}

\begin{proof}
First we show that the estimate is true if $-\delta_j<\delta<\delta_j$. Multiply equation \eqref{problem-K} by $e^{2\delta \tau}\omega$:
\be\label{estimate80}
- \int_0^\infty \omega(\partial_{\tau\tau}\omega) e^{2\delta \tau}\,d\tau
+\delta_j^2\int_0^\infty \omega^2 e^{2\delta \tau}\,d\tau=-\int_0^\infty \psi \omega e^{2\delta \tau}\,d\tau.\ee
Integration by parts, noting that the boundary terms vanish, yields
\be\label{estimate90}-\int_0^\infty \omega(\partial_{\tau\tau}\omega) e^{2\delta \tau}\,d\tau= \int_0^\infty
(\partial_\tau \omega)^2 e^{2\delta \tau}\,d\tau+ 2\delta\int_0^\infty \omega(\partial_\tau \omega)
e^{2\delta \tau}\,d\tau.\ee
For the last term in \eqref{estimate90},
\be\label{estimate95}2\delta\int_0^\infty \omega(\partial_\tau \omega) e^{2\delta \tau}\,d\tau=
\delta\int_0^\infty \partial_\tau (\omega^2)e^{2\delta \tau}\,d\tau =-2\delta^2 \int_0^\infty \omega^2
e^{2\delta \tau}\,d\tau.\ee
Substitute the two expressions above into \eqref{estimate80} to obtain
\be\label{estimate110}
\bs\int_0^\infty (\partial_\tau \omega)^2 e^{2\delta \tau}\,d\tau
+(\delta_j^2-2\delta^2) \int_0^\infty \omega^2 e^{2\delta \tau}\,d\tau  =-\int_0^\infty \psi \omega
e^{2\delta \tau}\,d\tau \\
\leq \lp\int_0^\infty \psi^2 e^{2\delta \tau}\,d\tau\rp^{\frac{1}{2}}
\lp\int_0^\infty \omega^2 e^{2\delta \tau}\,d\tau\rp^{\frac{1}{2}}.
\end{split}\ee
On the other hand, Holder estimates in \eqref{estimate95} above will give
$$\delta\int_0^\infty \omega^2 e^{2\delta \tau}\,d\tau\leq
 \lp \int_0^\infty \omega^2 e^{2\delta \tau} \,d\tau\rp^{\frac{1}{2}}
\lp\int_0^\infty (\partial_\tau \omega)^2 e^{2\delta \tau} \,d\tau\rp^{\frac{1}{2}}$$ and thus,
\be\label{estimate100} \delta^2\int_0^\infty \omega^2 e^{2\delta \tau}\leq \int_0^\infty (\partial_\tau
\omega)^2 e^{2\delta \tau}\, d\tau.\ee
Substituting \eqref{estimate100} into \eqref{estimate110} implies
\begin{equation*}\label{estimate115}
(\delta_j^2-\delta^2) \int_0^\infty \omega^2 e^{2\delta \tau}\,d\tau
\leq \lp\int_0^\infty \psi^2 e^{2\delta \tau}\,d\tau\rp^{\frac{1}{2}} \lp\int_0^\infty \omega^2
e^{2\delta \tau}\,d\tau\rp^{\frac{1}{2}}.
\end{equation*}
To finish, just note that $\delta_j^2-\delta^2>0$ because of our hypothesis, so that
$$\int_0^\infty \omega^2 e^{2\delta \tau}\,d\tau
\leq \frac{1}{(\delta_j^2-\delta^2)^2}\int_0^\infty \psi^2 e^{2\delta \tau}\,d\tau,$$
as desired.\\

Now we prove estimate \eqref{not-so-easy-preliminary} if $\delta>\delta_j$ (the remaining case $\delta<-\delta_j$ is very similar).
First remark that problem \eqref{problem-K} is an ODE, which has a unique solution. Using the variation of constants formula, it is written as
$$ \omega=\frac{1}{W}\lp B^+\int_{\tau}^{+\infty} B^- \psi - B^-\int_{\tau}^{+\infty} B^+ \psi \rp:=\frac{1}{W}(\omega_1+\omega_2),$$
where $B^+(\tau)=e^{\delta_j\tau}$ and $B^-(\tau)=e^{-\delta_j\tau}$  and $W$ the  Wronskian. We proceed
as follows: first, for the term $\omega_2:=B^- \int B^+ \psi$, use integration by parts
\begin{equation*}\begin{split}
\|\omega_2\|^2_{L^2_{\delta}(d\tau)}&= \int_{0}^{\infty} e^{-2\delta_j \tau} \lp \int_{\tau}^\infty
e^{\delta_j \tau} \psi\rp^2e^{2\tau\delta}\,d\tau\\
&=\int_0^\infty \partial_\tau \left( \frac{1}{2\delta-2\delta_j}e^{-2\delta_j\tau+2\delta \tau}\right)\lp \int_{\tau}^\infty
e^{\delta_j \tau} \psi\rp^2\,d\tau\\&
=\frac{1}{\delta-\delta_j}\int_0^\infty \psi \omega_2 e^{2\delta \tau}\,d\tau.
\end{split}\end{equation*}
To finish, just use Holder inequality:
$$\|\omega_2\|^2_{L^2_{\delta}(d\tau)}\lesssim \lp \int_0^\infty  \psi^2 e^{2\delta \tau}d\tau \rp^{\frac{1}{2}}
\|\omega_2\|_{L^2_{\delta}(d\tau)}.$$
For the first term of $\omega_1$ the inequality is proved similarly.
\end{proof}

Now we go back to the problem
\begin{equation}\label{problem-Kj}
K_j \omega=\psi\quad \mbox{for}\quad K_j \omega=\partial_{\tau\tau}\omega-\delta_j^2\omega-e^{-2\tau}\omega,
\end{equation}
in order to understand the dependence on $\tau_0$. While estimate \eqref{not-so-easy-preliminary} should still be true, the constant $C$ would depend on $\tau_0$. This is not enough to go back to the variable $t$ (recall that we are working in general with functions supported on $\mathcal T_\rho(\Lambda)$. The strongest assumption $\delta>0$ will provide this extra control, as we will show in the following Lemma:

\begin{lem}\label{lemma:tau0}
Fix $\delta>0$ not an indicial root. Let $\omega(\tau)$ be a solution of \eqref{problem-Kj} supported on $\tau\in(\tau_0,\infty)$, $\tau_0\in\mathbb R$. Then
\begin{equation}\label{not-so-easy-s}
\|\omega\|_{L^2_\delta(d\tau)}\leq C\|\psi\|_{L^2_\delta(d\tau)},
\end{equation}
for a constant $C$ independent of $\omega$ and $\tau_0$.
\end{lem}

\begin{proof}
The proof goes similarly to that of Lemma \ref{lemma:not-so-easy-preliminary}. First, in the case
$0<\delta<\delta_j$, one can repeat the proof line by line, noting that the additional term $-\int \omega^2e^{2\tau}e^{2\delta \tau}$ has the right sign and can be dropped while keeping the inequality.

The case $\delta>\delta_j$ is more delicate, since involves Bessel functions.  We can still write the solution to problem \eqref{problem-K} as
$$ \omega=\frac{1}{W}\lp B^+\int_{\tau}^{+\infty} B^- \psi - B^-\int_{\tau}^{+\infty} B^+ \psi \rp=:\frac{1}{W}(\omega_1+\omega_2),$$
where
\begin{equation*}
B_+(\tau):= K_{\delta_j}(e^{-\tau}),\quad
B_-(\tau):= I_{\delta_j}(e^{-\tau}),
\end{equation*}
where $I_{\delta_j},K_{\delta_j}$ are the modified Bessel functions of the second kind. Their asymptotic behavior is well known and, indeed, when $\tau\to +\infty$,
$B^+(\tau)\sim e^{{\delta_j} \tau}$ and $B^-(\tau)\sim e^{-{\delta_j} \tau}$. The  Wronskian $W$ is well known and has constant value (see \cite{Watson}, Chapter III, formula (80)).

We will give the proof for the term $\omega_1:=B_+ \int_\tau^\infty B_- \psi$. An analogous argument yields the estimate for $\omega_2$. First use integration by parts
\begin{equation*}\begin{split}
&\|\omega_1\|^2_{L^2_{\delta}(d\tau)}\\&= \int_{\tau_0}^{\infty} B_+(\tau)^2 \lp \int_{\tau}^\infty
B_-(\sigma) \psi(\sigma)\,d\sigma\rp^2 e^{2\tau\delta}\,d\tau\\
&=\int_{\tau_0}^\infty \partial_\tau \left( \frac{1}{2\delta+2\delta_j}e^{2\delta \tau +2\delta_j \tau}\right)e^{-2\delta_j \tau}B_+(\tau)^2\lp \int_{\tau}^\infty
B_-(\sigma) \psi(\sigma)\,d\sigma\rp^2\,d\tau\\&
=:J_1+J_2,
\end{split}\end{equation*}
where
\begin{equation*}
\begin{split}
J_1&=\frac{1}{\delta+\delta_j}\int_{\tau_0}^\infty e^{2\delta \tau}B_+(\tau)^2B_-(\tau)\psi(\tau)\int_\tau^\infty B_-(\sigma)\psi(\sigma)\,d\sigma \,d\tau\\&
=\frac{1}{\delta+\delta_j}\int_{\tau_0}^\infty e^{2\delta \tau}B_+(\tau)B_-(\tau)\psi(\tau) \omega_1(\tau)\,d\tau,
\end{split}
\end{equation*}
just nothing that $B_+(\tau)B_-(\tau)$ is a uniformly bounded positive function. Finally,  H\"older's inequality yields
\begin{equation*}
\begin{split}
J_1&\leq  \frac{C}{\delta+\delta_j} \left(\int_{\tau_0}^\infty \psi(\tau)^2e^{2\delta \tau}\,d\tau\right)^{1/2}\left(\int_{\tau_0}^\infty \omega_1(\tau)^2e^{2\delta \tau}\,d\tau\right)^{1/2}\\
& =\frac{C}{\delta+\delta_j}\|\psi\|_{L^2_{\delta}(d\tau)}\|\omega_1\|_{L^2_{\delta}(d\tau)}.
\end{split}\end{equation*}
On the other hand,
\begin{equation*}\begin{split}
J_2&= -\frac{1}{2\delta+2\delta_j}\int_{\tau_0}^\infty e^{2\delta \tau+2\delta_j \tau} \partial_\tau \left( e^{-2\delta_j \tau}B_+(\tau)^2\right)\int_\tau^\infty B_-(\sigma)\psi(\sigma)\,d\sigma \,d\tau \\
&= -\frac{1}{2\delta+2\delta_j}\int_{\tau_0}^\infty e^{2\delta \tau} \partial_\tau \log\left( e^{-2\delta_j \tau}B_+(\tau)^2\right)\omega_1(\tau)^2\,d\tau.
\end{split}
\end{equation*}
Calculate, for $s=e^{-\tau},$
\begin{equation*}
\begin{split}
\partial_\tau \log\left( e^{-2\delta_j \tau}B_+(\tau)^2\right)&=2\left[-\delta_j+\frac{\partial_\tau B_+(\tau)}{B_+(\tau)}\right]=2\left[-\delta_j-\frac{s\partial_s K_{\delta_j}(s)}{K_{\delta_j}(s)}\right]\\
&=
-2\frac{s^{1-\delta_j} \partial_s(s^{\delta_j} K_{\delta_j}(s))}{K_{\delta_j}(s)}\geq 0,
\end{split}\end{equation*}
using Property 3.71 in \cite{Watson} which implies  $\partial_s(s^{\delta_j} K_{\delta_j}(s))\leq 0$.
By the crucial hypothesis $\delta>0$, the term $J_2$ has a sign and can be dropped. From the estimate for $J_1$ we have that
$$\|\omega_1\|^2_{L^2_{\delta}(d\tau)}\leq C\|\psi\|^2_{L^2_{\delta}(d\tau)},$$
for functions supported in $(\tau_0,+\infty)$ but now this constant $C$ is independent of $\tau_0$, as desired.\\
\end{proof}

Finally, taking Fourier transform back will complete the proof of Proposition \ref{prop-not-so-easy}. If $\re\delta_0=0$, then one needs to consider Bessel functions with complex argument. Nevertheless, modifications are minimal.

\section{Injectivity of the model operator}\label{section:injectivity}

Let $\mu$ be a weight satisfying
\begin{equation}\label{choose-mu}
\tfrac{p}{2}-\tfrac{n-4}{4}\leq \re(\chi_{0,+}^{(0)})<\mu<\chi_{1,+}^{(0)}.
\end{equation}
We let $\mathcal L_1$ to be the variable coefficient operator which is given by \eqref{mathcalL} evaluated at $\ep=1$, that is,
\begin{equation}\label{operator:L1}
\mathcal L_1[\varphi]=\mathcal A_0^1(r) \varphi +  \frac{\mathcal A_1^1(r)}{r}\partial_r  \varphi +A_2^1(r)\partial_{rr} \varphi+\frac{\mathcal A_3^1(r)}{r^2}\Delta_\theta \varphi+\mathcal A_4^1(r) \Delta_z \varphi.
\end{equation}
Note that it is equivalent to work with $\mathcal L_1$ or the normalized $\mathcal L'_1$.

The aim of this section is to prove:

\begin{pro}\label{L1-injective}
The only solution $\varphi\in\mathcal C^{2,\alpha}_{\mu,0}(\mathbb R^n\setminus\mathbb R^p)$ of
\begin{equation}\label{equation50}
\mathcal L_1 \varphi=0\quad \text{in }\mathbb R^n\setminus\mathbb R^p
\end{equation}
is $\varphi\equiv 0$.
\end{pro}

Before we give the proof of the Proposition, let us study first the limiting behavior as $r\to 0$ and $r\to \infty$.

\subsection{The normal operators $\mathcal L^{(0)}$ and $\mathcal L^{(\infty)}$}

Consider the constant coefficient operators $\mathcal L^{(0)}$ and $\mathcal L^{(\infty)}$ on $\mathbb R^n\setminus \mathbb R^p$;  precise formulas were given in \eqref{L0} and \eqref{Linfty}.

\begin{pro}\label{model-injective} Any solution $\varphi\in\mathcal C^{2,\alpha }_{\mu,0}(\mathbb R^n\setminus \mathbb R^p)$ of $\mathcal L^{(0)}  \varphi=0$ must vanish identically.
\end{pro}

\begin{proof}
First remark that it is enough to study injectivity of the projected operators
\begin{equation*}
\mathcal L^{(0)}_j[\varphi]=\beta_0^{(0)} \varphi +  \frac{\beta_1^{(0)}}{r}\partial_r  \varphi +\beta_2^{(0)}\partial_{rr} \varphi-\lambda_j\frac{\beta_3^{(0)}}{r^2} \varphi+\beta_4^{(0)} \Delta_z \varphi=0.
\end{equation*}
Recalling the discussion in Section \ref{section:Fredholm}, this can be reduced to proving injectivity for each equation
\begin{equation}\label{eq:j}
\partial_{\tau\tau}\omega-\delta_j^2\omega-e^{-2\tau}\omega=0, \quad \text{for }\omega=\omega(\tau), \quad\tau\in\mathbb R, \quad j=0,1,\ldots
\end{equation}
under the assumption that $\omega$ has the behavior
\begin{equation*}
\omega(r)=\left\{\begin{split}
&O(r^{\updelta})\quad\text{as}\quad r\to 0,\\
&O(r^{-\frac{p}{2}+\frac{n-4}{4}}) \quad\text{as}\quad r\to\infty.
\end{split}\right.
\end{equation*}
Here we have defined
$$\updelta:=\mu-\tfrac{p}{2}+\tfrac{n-4}{4}>0.$$

Let us look then at equation \eqref{eq:j}. For any $j$, this is a Bessel ODE which has two linearly independent solutions, given by the modified Bessel functions of second kind $K_{\delta_j}(r)$ and $I_{\delta_j} (r)$ in the variable $r=e^{-t}$. Since the asymptotic behavior of the Bessel functions is well known, any solution with such behavior as $r\to 0$ cannot be bounded as $r\to\infty$, which is not possible because the choice of function space.
\end{proof}

\begin{pro}\label{model-injective-infty}
Similarly, any solution $\varphi\in\mathcal C^{2,\alpha }_{\mu,0}(\mathbb R^n\setminus \mathbb R^p)$ of
$$\mathcal L^{(\infty)}  \varphi=0$$
 must vanish identically.
\end{pro}

\begin{proof}
It is the same proof as in Proposition \ref{model-injective}, using the asymptotics of the Bessel functions, but with the new indicial roots. First, for each $j$, there are two solutions. However, the one that is not exponentially growing as $r\to \infty$ is not in the kernel thanks to condition \eqref{indicial-infinity-first}.
\end{proof}

\subsection{Beginning of the proof of Proposition \ref{L1-injective}}

The first observation is that, since $U_\epsilon$ only depends on the radial variable, the coefficients $\mathcal A_\ell^1$, $\ell=0,1,2,3,4$ only depend on $r$ as well, so one can project over spherical harmonics and show injectivity for each operator
\begin{equation*}\mathcal L_{1,j} \varphi=0,\quad \varphi=\varphi(r,z), \,r>0,z\in\mathbb R^p,
\end{equation*}
for $j=0,1,\ldots$. Here
\begin{equation}\label{operator:L1j}
\mathcal L_{1,j}\varphi=\mathcal A_0^1(r) \varphi +  \frac{\mathcal A_1^1(r)}{r}\partial_r  \varphi +A_2^1(r)\partial_{tt} \varphi-\lambda_j\frac{\mathcal A_3^1(r)}{r^2}\varphi+\mathcal A_4^1(r) \Delta_z \varphi.
\end{equation}

Next, as we did in Section \ref{section:Fredholm}, Fourier transform in the variable $z$ reduces to the problem to study the operator
\begin{equation*}
 J_j \omega:=\mathcal A_0^1(r) \omega +  \frac{\mathcal A_1^1(r)}{r}\partial_r  \omega +A_2^1(r)\partial_{rr} \omega-\lambda_j\frac{\mathcal A_3^1(r)}{r^2}\omega-\mathcal A_4^1(r) |\zeta|^2 \omega=0.
\end{equation*}
This is an ODE in the variable $r\in\mathbb R$ for each fixed $\zeta$.  Indicial roots for this problem were given in Proposition \ref{prop:indicial}.
We will consider the different values of $j$ in the following paragraphs, and conclude the proof in Subsection \ref{subsection:higher-modes}.

The first observation is that, for $j=0$, our choice of weight $\mu>\re(\chi_{0,+}^{(0)})$ prevents having any solution in the kernel with behavior $O(r^\mu)$ as $r\to 0$.\\

Next, in  Sections  \ref{subsection:radial-nondegeneracy} and \ref{subsection:rotational-nondegeneracy} we try to understand the effect of symmetries of the equation. It will be useful to consider  the reduced operator (when $p=0$  so there is no variable $z$) and $\ep=1$, which is given by
\begin{equation*}
\mathfrak L_{1}\varphi=\mathcal A_0^1(r) \varphi +  \frac{\mathcal A_1^1(r)}{r}\partial_r  \varphi +\mathcal A_2^1(r)\partial_{rr} \varphi+\frac{\mathcal A_3^1(r)}{r^2}\Delta_\theta\varphi,
\end{equation*}
and its spherical harmonic projection
\begin{equation*}
\mathfrak L_{1,j}\varphi=\mathcal A_0^1(r) \varphi +  \frac{\mathcal A_1^1(r)}{r}\partial_r  \varphi +\mathcal A_2^1(r)\partial_{rr} \varphi-\lambda_j\frac{\mathcal A_3^1(r)}{r^2}\varphi.
\end{equation*}

\subsection{Non-degeneracy in the radial direction}\label{subsection:radial-nondegeneracy}

Even if it is not needed in our discussion, let us take a detour to characterize the kernel of the operator $\mathfrak L_{1,0}$ (that is, for $j=0$), given by
\begin{equation*}
\mathfrak L_{1,0}[\varphi]=\mathcal A_0^1(r) \varphi +  \frac{\mathcal A_1^1(r)}{r}\partial_r  \varphi +\mathcal A_2^1(r)\partial_{rr} \varphi, \quad\varphi=\varphi(r),
\end{equation*}
and prove that it is non-degenerate. We start with an immediate observation:

\begin{remark}\label{remark:dilation-inv}
The $\sigma_2$-equation is dilation invariant. This implies that the function
$$\varphi_\sharp:=r\partial_r U_1+\tfrac{n-4}{4}U_1$$
 is a solution to the linear problem  $\mathfrak L_{1,0} \varphi_\sharp=0$.
\end{remark}

We will show in the next Lemma that this is actually the only possible solution. For this, it is better to go back to the tilde-notation consider the conjugate operator $\widetilde{\mathfrak L}_{1,0}$ and the corresponding solution $\widetilde w_\sharp$.

\begin{lemma}
Any other solution to
\begin{equation}\label{eq:nondegeneracy}
\mathfrak L_{1,0}\widetilde w=0
\end{equation}
that decays to zero as  $t\to \infty$ must be a multiple of $\widetilde w_\sharp$.
\end{lemma}

\begin{proof}
Let $\widetilde w_1,\widetilde w_2$ be two solutions of \eqref{eq:nondegeneracy} that decay to zero as $t\pm\infty$, that is, of the form
$\widetilde w_i=\alpha_i (1+o_i(1))e^{\varsigma_i t}$, $i=1,2$, as $t\to +\infty$, $\alpha_i\neq 0$.
Define its Wronskian $W(t)$ by
$$W(t)=\widetilde w_1'\widetilde w_2-\widetilde w_1 \widetilde w_2'=\alpha_1 \alpha_2(\varsigma_1-\varsigma_2+o(1))e^{(\varsigma_1+\varsigma_2)t}.$$
Since the Wronskian is constant, then we must have $\varsigma_1=\varsigma_2$. Let us assume, by rescaling, that $\alpha_1=\alpha_2=1$. Now look at the next order. For this, we write
$\widetilde w_i=e^{\varsigma_i t}+\alpha_i^1(1+o_i(1))e^{\varsigma_i^{(1)} t} $, $i=1,2$, as $t\to +\infty$, $\alpha_i^{(1)}\neq 0$. The same argument will yield that $\varsigma_1^{(1)}=\varsigma_2^{(1)}$.
 Inductively, we will be able to conclude that $w_1\equiv w_2$, since  we have analytic continuation for the solutions of this ODE.
\end{proof}

\subsection{Rotational invariance implies non-degeneracy}\label{subsection:rotational-nondegeneracy}

Now we use rotational invariance to show injectivity for the first non-zero mode, this is, for $\lambda_1=\ldots=\lambda_N=N-1$ (recall the notation in \eqref{spherical-harmonics}). Assume that $p=0$ for now and study $\mathfrak L_{1,j}$ for $j=1,\ldots,N$.

\begin{remark}\label{remark:rotation-invariance}
Notice first that  rotational invariance yields that
\begin{equation*}
\mathfrak L_1 (\partial_j U_1)=0, \quad j=1,\ldots,N.
\end{equation*}
Since $\partial_j = e_j \partial_r$, then $\varphi_\diamond:=\partial_r U_1$ belongs to the kernel of $\mathfrak L_{1,j}$ for each $j=1,\ldots,N$.
Moreover,
\begin{equation*}
\varphi_\diamond(r)\asymp
r^{-\frac{n-4}{4}-1}\quad\text{as}\quad r\to 0,
\end{equation*}
and
\begin{equation*}
\varphi_\diamond(r)\sim r^{-\alpha_0-\frac{n-4}{4}-1} \quad\text{as}\quad r\to\infty.
\end{equation*}
In particular, this shows \eqref{chi1-}.
\end{remark}

By contradiction, assume that $\varphi_j$ is a solution to $\mathfrak L_{1,j}\varphi_j=0$ in the space $\mathcal C^{2,\alpha }_{\mu,0}$. Looking at \eqref{chi1+}, one knows that it behaves like
\begin{equation*}
\varphi_j(r)\sim r^{-\alpha_0-\frac{n-4}{4}-1}\quad \text{as}\quad r\to\infty.
\end{equation*}
We also have, by our choice of $\mu$ in \eqref{choose-mu}, that
\begin{equation*}
\varphi_j\sim r^{\chi^{(0)}_{j,+}}\quad\text{as}\quad r\to 0.
\end{equation*}
Then there is a (non-trivial) linear combination of $\varphi_\diamond$ and $\varphi_j$ that decays faster than $r^{-\alpha_0-\frac{n-4}{4}-1}$ as $r\to\infty$. Looking at the different behaviors as $r\to 0$, we see that this combination is non-vanishing. Looking at the indicial roots $\chi^{(\infty)}_{j,\pm}$,  this is a contradiction since no (non-trivial) solution can decay faster at $r\to\infty$.

Now, to pass from $\mathfrak L_{1,j} $ to $\mathcal L_{1,j}$, $j=1,\ldots, N$ we need to take Fourier transform in $z$ and use the same continuity argument as in \cite[Proposition 4]{Mazzeo-Pacard}, considering the Fourier variable $|\zeta|^2$ as a parameter.

\subsection{The higher modes}\label{subsection:higher-modes}
To complete the proof of Proposition \ref{L1-injective} it remains to study the case  $j>N$.

We consider first the eigenvalue problem for $\mathcal L_{1,0}\varphi=\eta\varphi$. Although we know that 0 is an eigenvalue,  we have no information on its Morse index. Let $\eta_0$ be the first eigenvalue. It is well known (thanks to self-adjointness with respect to the scalar product \eqref{space-H}), that its corresponding eigenfunction $\varphi_0$ is strictly positive.

Now, let $\varphi_j$ be a solution to $\mathcal L_{1,j}\varphi_j=0$, for $j>N$, which can be written as $\mathcal L_{1,0}\varphi_j=\lambda_j\varphi_j$ after we have taken Fourier transform in the variable $z$. Combining this equation with $\mathcal L_{1,0}\varphi_0=\eta_0\varphi_0$ we arrive to
\begin{equation*}
(\lambda_j-\eta_0)\varphi_j\varphi_0=\varphi_0 \mathcal L_{1,0} \varphi_j-\varphi_j \mathcal L_{1,0}\varphi_0.
\end{equation*}
Integrate this expression in the set where $\{\varphi_j>0\}$ with respect to the weighted $L^2$ space from \eqref{space-H}, and use the divergence theorem to obtain that
\begin{equation*}
\begin{split}
(\lambda_j-\eta_0)&\int_{\{\varphi_j>0\}}\varphi_j\varphi_0 \mathcal H^1 \,dr=\int_{\{\varphi_j=0\}} \mathcal H^1\left\{\varphi_0 \partial_{\vec\nu } \varphi_j- \varphi_j \partial_{\vec\nu} \varphi_0\right\}\,ds.
\end{split}
\end{equation*}
Here $\vec\nu$ is the exterior normal to the integration set. It is easy to check that $\partial_{\vec\nu} \varphi_j<0$. Then, from the above formula we reach a contradiction unless
$$\int_{\{\varphi_j>0\}}\varphi_j\varphi_0 \mathcal H^1\,dr=0.$$
In particular, this implies that $\varphi_j\equiv 0$, as desired.

\section{Injectivity of $\mathbb L_\epsilon$}

Let $\mu$ be a weight as in \eqref{choose-mu}.  We will show first injectivity in H\"older spaces. Again, it is equivalent to work with $\mathbb L'_\ep$ or $\mathbb L_\ep$.

\begin{pro}\label{prop:uniform-injectivity}
There exists $\epsilon_0$ such that for all $0<\ep<\ep_0$, the operator $\mathbb L'_\ep$ is injective in $\mathcal C^{2,\alpha}_\mu (M\setminus\Lambda)$.
\end{pro}

The idea is to use a contradiction argument as $\ep\to 0$ which, after rescaling, reduces the problem to the model operator $\mathcal L'_1$. This is a rather standard argument by now (see \cite{Mazzeo-Pacard} or \cite[Proposition 3.1]{Fakhi} for the scalar curvature case). Thus, assume that there exists a sequence $\{\ep_l\}$, $\ep_l\to 0$ such that
$\mathbb L'_l:=\mathbb L'_{\epsilon_l}$ is not injective, i.e., there exists
$\varphi_l\in\mC^{2,\alpha}_\mu$ with $\mathbb L'_l \varphi_l=0$. Rescaling, we can always assume that
$$\|\varphi_l\|_{\mC^{0}_\mu}=1.$$
Then there exists $y_l\in M\setminus\Lambda$ such that
\be\label{sup-achieved}1\geq d_l^{-\mu}\varphi_l(y_l)>\tfrac{1}{2},\ee
where $d_l:=d_\Lambda(y_l)$ and $d_\Lambda$ is the distance function defined in Definition \ref{def:holder-spaces}.
Up to a subsequence, we can find $y_0\in M$ such that $y_l\to y_0$.\\

We will need a preliminary convergence Lemma:

\begin{lem} \label{lemma:convergence}
Consider  $\mu>\frac{p}{2}-\frac{n-4}{4}$ not an indicial root. Given a sequence $\{\varphi_l\}$ in ${\mC^{2,\alpha}_\mu}$, if $\|\varphi_l\|_{\mC^{2,\alpha}_\mu}\leq C$, then
it has a convergent subsequence. This is still true if we only have
$\|\varphi_l\|_{\mC^\alpha_{\mu}}\leq C $ but each $\varphi_l$ is a solution of the homogeneous equation $\mathbb L'_l \varphi=0$.
\end{lem}

\begin{proof}
We use elliptic estimates and Ascoli's theorem in compact sets. Note that, if we have convergence of a subsequence in a compact set, the estimate \eqref{not-so-easy} will imply convergence in weighted $L^2$-spaces defined on the full domain $M\setminus\Lambda$. Weighted elliptic estimates again will yield the desired conclusion.
\end{proof}

There are several possibilities according to the position of $y_0$:\\

\textbf{Case 1:} we first assume that $y_0\not\in \Lambda$. 
For $\ep_l$ small enough, the operator $\mathbb L'_l$ converges $\mathbb L'$ where  $\mathbb L'=\mathbb L(1,g_M)$, the background operator on $M$, or, if $y_0$ belongs to the transition region from Proposition \ref{prop-gluing-Lambda}, then $\mathbb L'=\mathbb L(u_*,g_M)$. In any case, both are non-degenerate.

By Lemma \ref{lemma:convergence}, the sequence $\{\varphi_l\}$  converges (up to passing to a subsequence) to some $\varphi_0$ in $M\setminus\Lambda$ satisfying
\begin{equation}\label{equation60}\mathbb L' \varphi_0=0\quad\text{in }M\setminus \Lambda.\end{equation}
Then, since $\Lambda$ has Hausdorff dimension $p<n-2$ and $u=O(d_\Lambda^\mu)$,  we can apply classical removability of singularity results for quasilinear equations (see, for instance, Chapter 3.1.3 in \cite{Maly-Ziemer}, or the original \cite{Serrin}) to conclude that $\varphi_0$ can be extended to a weak solution of $\mathbb L' \varphi_0=0$ in the whole $M$. Then we must have $\varphi_0\equiv 0$ due to  non-degeneracy. This yields a  contradiction with \eqref{sup-achieved}.\\

\textbf{Case 2:} Assume now that $y_0\in\Lambda$. Let $r_l=\dist(x_l,\Lambda)$ to be the radial Fermi coordinate, and rescale
$$ \hat\varphi_l (y)=r_l^{-\mu}\varphi_l(y_l+r_l y).$$
Then Lemma \ref{lemma:convergence} implies that, up to a subsequence, $\hat \varphi_l$
converges to $\hat\varphi_0\in\mC^\alpha_{\mu,0}(\mbR^n\back\mbR^p)$, a solution of $L'\hat \varphi_0=0$
in $\mbR^n\back\mbR^p$ for some linear operator $L'$. In particular, from \eqref{sup-achieved} we know that 
\be\label{sup-achieved1}1\geq \hat\varphi_0(y_0)\geq\tfrac{1}{2}.\ee

There are several possibilities for $L'$ according to the behavior of $\ep_l/r_l$, since the approximate solution $\bar u_\ep$ has a different behavior in each regime from Corollary \ref{cor-model-solution}:

First, if $\ep_l/r_l\to +\infty$, then we are in the situation \eqref{behavior-zero}, and
$L'$ reduces to be the operator $\mathcal L^{(0)}$ from \eqref{L0} in $\mathbb R^n\setminus \mathbb R^p$.
Because of Proposition \ref{model-injective}, $\hat\varphi_0$ must be identically zero. Contradiction again with \eqref{sup-achieved1}.

Second, assume  $\ep_l/r_l\to C\neq 0$. Without loss of generality, we can take $C=1$, otherwise rescale. Then, taking into account the scaling \eqref{rescale}, $L'$ coincides with the operator $\mathcal L'_1$ in expression \eqref{operator:L1}. We can use Proposition \ref{L1-injective} to reach a contradiction as above.

Finally, if $\ep_l/r_l\to 0$,  the operator $L'$ is simply $\mathcal L^{(\infty)}$ from \eqref{Linfty} defined in $\mathbb R^n\setminus\mathbb R^p$, and we argue as in the previous cases, using Proposition \ref{model-injective-infty} to conclude.

     This completes the proof of Proposition \ref{prop:uniform-injectivity}.\\
\qed

For injectivity in Lebesgue spaces, first set
 \begin{equation}\label{definition-delta}
 \delta=\mu-\tfrac{p}{2}+\tfrac{n-4}{4}.
\end{equation}
The shift $-\tfrac{p}{2}+\tfrac{n-4}{4}$ allows to pass from the original operator $\mathbb L'_\ep$, where we have injectivity in H\"older spaces, to the conjugated $\widetilde{\mathbb L}'_\ep$, where Hilbert space analysis is  more clear.

\begin{cor}\label{cor-injectivity}
For every $\ep\in(0,\ep_0)$, $(\widetilde{\mathbb L}'_\ep)^*$ is
injective in $L^2_{\delta}(M\setminus\Lambda)$, and thus, $\widetilde{\mathbb L}'_\ep$ is surjective in $L^2_{-\delta}(M\setminus\Lambda)$.
\end{cor}

\begin{proof}
We use first Proposition \ref{prop:uniform-injectivity} to obtain injectivity in H\"older spaces.
Then, classical regularity estimates allow to pass from Lebesgue spaces to H\"older spaces. This implies, in particular, that injectivity holds in weighted $L^2_\delta$ spaces if we choose  a parameter as in  \eqref{definition-delta}.

For the second assertion, simply recall the relation \eqref{ker}.
\end{proof}





Finally, we will prove an auxiliary result that will be needed in the next section:

\begin{lemma}\label{adjoint-bound}
Assume $\mu$ satisfies \eqref{choose-mu}. There exists $\epsilon_0>0$ and $C>0$ such that, for every $\epsilon\in(0,\ep_0)$, if $\psi\in\mathcal C^{2,\alpha}_{\mu}(M\setminus\Lambda)$ is a solution to
\begin{equation*}
(\mathbb L'_\ep)^* \psi= h,
\end{equation*}
for $h\in\mathcal C^{0,\alpha}_{\mu-2}(M\setminus\Lambda)$, then
\begin{equation}
\|\psi\|_{\mathcal C^{2,\alpha}_\mu}\leq C\|h\|_{\mathcal C^{0,\alpha}_{\mu-2}}.
\end{equation}
\end{lemma}

\begin{proof}
This a contradiction argument very similar to the proof of Proposition \ref{prop:uniform-injectivity} and thus, we omit it.
\end{proof}

\section{Uniform surjectivity of $\mathbb L_\epsilon$}\label{section:uniform-surjectivity}

We have just shown that for each fixed $\ep$, the operator
$$\widetilde{\mathbb L}'_\ep:L^2_{-\delta}(M\setminus\Lambda)\to L^2_{-\delta}(M\setminus\Lambda)$$
 is surjective.
Now we would like to construct a right inverse for $\mathbb L'_\epsilon$.
For that we set
$$\widetilde{\mathds{ L}}'_\epsilon:=\widetilde{\mathbb L}'_\epsilon \circ d_\Lambda^{2\delta} \circ (\widetilde{\mathbb L}'_\epsilon)^* :
L^2_{-\delta}(M\setminus\Lambda)\to L^2_{-\delta}(M\setminus\Lambda).$$
This $\widetilde{\mathds{ L}}'_\epsilon$ is an isomorphism and thus, it has a
bounded two-sided inverse
$$\widetilde{\mathds{G}}'_\epsilon:L^2_{-\delta}(M\setminus\Lambda)\to L^2_{-\delta}(M\setminus\Lambda).$$
In particular, $\widetilde{\mathds{L}}'_\epsilon \circ \widetilde{\mathds{G}}'_\epsilon =I$, which means that
\begin{equation}\label{duality2}\widetilde{\mathbb{G}}'_\epsilon:=d_\Lambda^{-2\delta}\,(\widetilde{\mathbb L}'_\epsilon)^*\, d_\Lambda^{2\delta}\,\widetilde{\mathds{G}}'_\epsilon:
L^2_{-\delta}(M\setminus\Lambda)\to L^2_{-\delta}(M\setminus\Lambda)
\end{equation} is a bounded right inverse for $\widetilde{\mathbb L}_\epsilon$ that maps into the
range of $d_\Lambda^{-2\delta}(\widetilde{\mathbb L}'_\epsilon)^*$.\\

In particular, an analogous result is true for $\mathbb G'_\ep$ with the usual shift of indexes, analogous to \eqref{definition-delta}. Now we choose
\begin{equation}\label{choose-nu}\nu<\tfrac{p}{2}-\tfrac{n-4}{4}\quad\text{slightly larger than } \nu':=-\delta+\tfrac{p}{2}-\tfrac{n-4}{4}\end{equation}
and
  restrict this inverse to  the smaller set $\mathcal C^{0,\alpha}_{\nu- 2}(M\setminus\Lambda)$.
Then
\begin{equation*}
\mathbb G'_\ep:\mathcal C^{0,\alpha}_{\nu- 2}(M\setminus\Lambda)\to L^2_{\nu'}(M\setminus\Lambda).
\end{equation*}
Let us improve the regularity of this inverse:

\begin{pro}\label{prop:improvement}
If $\varphi\in L^2_{\nu'}(M\setminus\Lambda)$ is a solution of $\mathbb L'_\epsilon \varphi=h$ for $h\in \mathcal C^{0,\alpha}_{\nu- 2}(M\setminus\Lambda)$, then we have that $\varphi\in \mathcal C^{2,\alpha}_{\nu}(M\setminus\Lambda)$ for
$\nu$ close enough to $\nu'$, $\nu'<\nu$.
\end{pro}

\begin{proof}
Rescaled Schauder estimates immediately imply that the solution $\varphi\in \mathcal C^{2,\alpha}_{\nu'}(M\setminus\Lambda)$. However, the main statement in this Proposition is an improvement of weight from $\nu'$ to $\nu$. This fact follows from the work of Mazzeo \cite[Theorem 7.14]{Mazzeo:edge}, where they show that a change in the asymptotics of $\varphi$ is created only by the crossing of an indicial root, and this cannot happen if $\nu$ and $\nu'$ are close enough.
\end{proof}

In summary, there exists a (bounded) right inverse for
 $$\mathbb L'_\ep:\mathcal C^{2,\alpha}_{\nu}(M\setminus\Lambda)\to\mathcal C^{0,\alpha}_{\nu- 2}(M\setminus\Lambda),$$
 given  by
 $$\mathbb G'_\ep:\mathcal C^{0,\alpha}_{\nu- 2}(M\setminus\Lambda)\to \mathcal C^{2,\alpha}_{\nu}(M\setminus\Lambda).$$
 The main result in  this section is the following:

\begin{pro}\label{prop:unif-surjectivity}
There exists $\ep_0>0$ such that, for all $\ep\in(0,\ep_0)$, the norm of $\mathbb G'_\ep$ does not depend on $\ep$.
\end{pro}

\begin{proof}
The proof is very similar to \cite[Theorem 6]{Mazzeo-Pacard} and \cite[Proposition 4.4]{Fakhi}.
As in the proof of Proposition \ref{prop:uniform-injectivity}, we argue by contradiction. Assume that there exist sequences $\{\ep_l\}$,  $\{h_l\}$, $\{\varphi_l\}$ satisfying $\varphi_l=\mathbb G'_{\ep_l} h_l$ and
$$\sup_{M\setminus\Lambda} \{(d_\Lambda)^{-\nu} |h_l|\}=1\quad\text{but}\quad \sup_{M\setminus\Lambda} \{(d_\Lambda)^{-\nu} |\varphi_l|\}=:m_l\to\infty$$
as $\ep_l\to 0$.

We denote $\mathbb L'_l:=\mathbb L'_{\ep_l},\mathbb G'_l:=\mathbb G'_{\ep_l}$ and recall that, by the duality relation \eqref{duality2}, $\varphi_l= (d_\Lambda)^{2\nu}({\mathbb L}'_\ep)^*\psi_l$ for $\psi_l\in \mathcal C^{2,\alpha}_{\mu'}$ with $\mu'$ very close to $\mu$. Now
rescale
$$\hat \varphi_l:=\frac{\varphi_l}{m_l},\quad \hat h_l:=\frac{h_l}{m_l},\quad \hat \psi_l:=\frac{\psi_l}{m_l}.$$
Choose a point $y_l\in M\setminus\Lambda$ where
$$\tfrac{1}{2}\leq d_\Lambda(y_l)^{-\nu}\hat \varphi_l(y_l)\leq 1.$$
By compactness, we can show that, up to a subsequence, $y_l\to y_0$ for some $y_0\in M$. There are two cases depending on the location of $y_0$:\\

\textbf{Case 1:} Suppose that $y_0\in M\setminus\Lambda$.  Arguing as in Case 1 in the proof of Proposition \ref{prop:uniform-injectivity} we reach a contradiction using non-degeneracy.\\

\textbf{Case 2:} Now assume that $y_0\in \Lambda$, and let $r_l:=\dist(y_l,\Lambda)$ to be the radial Fermi coordinate, and rescale
$$\check\varphi_l(y):=r_l^{-\nu} \hat\varphi_l(y_l+r_l y),$$
and similarly with the other functions.
Since  $\|\check\varphi_l\|_{\mathcal C^{2,\alpha}_\nu}\leq C$, we can find a convergent subsequence to a function $\check \varphi_0$ satisfying
\begin{equation}\label{surjectivity1}
L' \check \varphi_0=0\quad\text{in }\mathbb R^n\setminus\mathbb R^p,
\end{equation}
for some operator $L'$ that we will study below. One may also check that $\check \psi_l$ converges to a function $\check \psi_0$, thanks to the bounds in Lemma \ref{adjoint-bound}. Moreover, this limit satisfies
\begin{equation}\label{surjectivity2}
\check \varphi_0=(L')^* (\check\psi_0),
\end{equation}
Take Fourier transform of equations \eqref{surjectivity1} and \eqref{surjectivity2} in the variable $z$, denoting the Fourier variable by $\zeta$, and the transformed operator by  $K_\zeta$. Then the both equations reduce to
\begin{equation*}
\begin{split}
&K_\zeta \omega_0=0,\\
&\omega_0=K_\zeta^*\varpi_0 \quad\text{in }\mathbb R^n\setminus\mathbb R^p.
\end{split}
\end{equation*}
Hence $0=K_\zeta (\omega_0)=K_\zeta K_\zeta^*  (\varpi_0)$. Multiply this equation by $\varpi_0$ and integrate by parts to obtain, for each fixed $\zeta$,
$$\int |K_\zeta^* \varpi_0|^2r^2\,dr=0.$$
This implies that $\varpi_0\equiv 0$, which yields a contradiction.


 The linear operator $L'$ (and its Fourier version $K$) will depend on the behavior of the quotient $\ep_l/r_l$, as we have three different regimes in Corollary \ref{cor-model-solution}. First, if $\ep_l/r_l\to \infty$, we have that $L'= \mathcal L^{(0)}$ because of \eqref{behavior-zero}.
Second, assume that $\ep_l/r_l\to 1$; then $L'=\mathcal L_1$ defined in $\mathbb R^n\back \mathbb R^p$.   Finally, if $\ep_l/r_l\to 0$, it holds that $L'=\mathcal L^{(\infty)}$. But in all these cases the above argument works, as it is very similar to the proof in Proposition \ref{prop:uniform-injectivity}.
\end{proof}

 \begin{remark}
In addition to \eqref{choose-nu}, we will impose further restrictions on $\nu$. Indeed, we need to ask that $\nu$ also satisfies
\begin{equation}\label{choose-nu2}
-\tfrac{n-4}{4}<\nu<\min\{-\tfrac{n-4}{4}+1,-\re(\chi^{(0)}_{0,-})\}.
\end{equation}
 Recalling the values of $\mu$ and $\delta$ in \eqref{choose-mu} and \eqref{definition-delta}, respectively, it is clear that these are compatible choices. A summary can be found in Figure 1.
 \end{remark}

\section{Nonlinear analysis}

Now we go back to our original equation, rewritten in cylindrical coordinates as \eqref{eq-t} in the notation of Subsection \ref{subsection:estimates}.
Thus consider equation \eqref{eq-tt}, which we recall here:
 \begin{equation}\label{eq-ttt}
\bar{L}_\ep[\varphi] +\bar f_\ep + \bar Q_\ep[\varphi]=0,
\end{equation}
We set
\begin{equation*}
\bar \nu=\nu+\tfrac{n-4}{4},
\end{equation*}
for some  $\bar \nu>0$ small enough. This choice is still compatible with \eqref{choose-nu2}. 

From the discussion in the previous section we have that the operator
 $$\bar L'_\ep:\mathcal C^{2,\alpha}_{\bar \nu}(M\setminus\Lambda)\to\mathcal C^{0,\alpha}_{\bar\nu}(M\setminus\Lambda),$$
has a right inverse
 $$\bar G'_\ep:\mathcal C^{0,\alpha}_{\bar \nu}(M\setminus\Lambda)\to \mathcal C^{2,\alpha}_{\bar\nu}(M\setminus\Lambda),$$
with norm independent of $\epsilon$ (see Proposition \ref{prop:unif-surjectivity}). Thus, $\bar L_\ep$ is also invertible, with right inverse
\begin{equation}\label{normalized-G}
\bar G_\ep:= \bar G'_\ep( \bar v_\ep^{-3} \,\cdot).
\end{equation}
Now we can define the operator
\begin{equation}\label{operatorT}
\bar T_\ep(w):=-\bar G_\ep(\bar Q_\ep[w]+\bar f_\ep).
\end{equation}
A fixed point of $\bar T_\ep$ will yield a solution to equation \eqref{eq-t}.\\

In our first claim we will use the estimates   for the error term $\bar f_\ep$ from Proposition \ref{prop:compare-metric}. Without loss of generality, take $\varrho<1$.
By \eqref{normalized-G} and the fact that $\bar G'_\ep$ has  norm uniformly bounded, 
\begin{equation*}
G_\ep\bar f_\ep=\begin{cases}
O(r),&\quad  r\in (0,\tfrac{1}{m}\ep),\\
O(\ep),&\quad r\in (\tfrac{1}{m}\ep,m\ep),\\
O(\ep^{\alpha_0} r^{-\alpha_0+1}),&\quad  r\in(m\ep,\varrho),\\
O(\ep^{\alpha_0}),&\quad r>\varrho,
\end{cases}
\end{equation*}
where the radial coordinate $r$ has been extended smoothly outside $\mathcal T_{\varrho}$. 
Set
\begin{equation}\label{fixed-point-region}
\begin{split}
\mathcal B(\ep,\varsigma):=\{w\in \mathcal C^{2,\alpha}_{\bar\nu}(M\setminus\Lambda)\,:\, &\|w\|_{\mathcal C^{2,\alpha}_{\bar\nu}(\mathcal T_{\varrho}) } \leq \varsigma \ep^{1-\bar\nu},\\
&\|w\|_{\mathcal C^{2,\alpha}(M\setminus\mathcal T_\varrho)}\leq \varsigma\ep^{\alpha_0}\}.
\end{split}
\end{equation}
 Choose $\varsigma$ large enough so that $\bar G_\ep \bar f_\ep\in \mathcal B(\ep,\varsigma/2)$. Next, 
 an estimate for the quadratic term follows by a standard argument. Indeed:

\begin{lem}
For $\epsilon_0$ small enough, we have
\begin{equation*}
\|\bar G_\ep \bar Q(w_2)-\bar G_\ep \bar Q(w_1)\|_{\mathcal C^{0,\alpha}_{\bar \nu}}\leq  \frac{1}{2}\|w_2-w_1\|_{\mathcal C^{2,\alpha}_{\bar \nu}}
\end{equation*}
for all $w_1,w_2\in\mathcal B(\ep_0,\varsigma)$, $\epsilon\in(0,\epsilon_0)$.
\end{lem}

\begin{proof}
This is just Lemma 5.2 in \cite{Fakhi} or Lemma 9 in \cite{Mazzeo-Pacard}, taking into account the  `prime' normalization.
\end{proof}

In conclusion, the fact that the operator  $\bar T_\ep$ from \eqref{operatorT} is a contraction on $\mathcal B(\ep,\varsigma)$ for $\ep$ small enough yields a solution $w$ to \eqref{eq-ttt}.

Going back to the $u$-notation for a moment, we have just produced a solution  $\bar u:=\bar u_\ep+\varphi$ such that $\bar u>0$ near $\Lambda$ and the metric $\bar g=u^{\frac{8}{n-4}}g_M$ is complete. Indeed, one the one hand, $\bar u_\ep$ behaves like $v_\infty r^{-\frac{n-4}{4}}$ as $r\to 0$ (see \eqref{behavior-zero}). On the other hand, thanks to the choice  $\nu>-\frac{n-4}{4}$ from \eqref{choose-nu2}, $\varphi=O(r^\nu)$ has a less singular behavior which is not seen near $\Lambda$. 

Unfortunately, our argument does not yield  positivity in the whole $M$ since the approximate solution is not sharp enough in the neck region $r\in(\varrho_0,\varrho_1)$. We conjecture that a more appropriately chosen approximate solution is possible, as in the paper \cite{cat-maz}, however, this question is out of the scope of this paper.

Family, observe that this proof also produces a family of solutions as $\ep\to 0$.

\qed

\medskip


\textbf{Financial support.}
 M.d.M. Gonz\'alez  acknowledges financial support from the Spanish  Government, grant numbers PID2020-113596GB-I00,  PID2023-150166NB-I00; additionally, Grant RED2022-134784-T, RED2024-153842-T funded by MCIN/AEI/ 10.13039/501100011033, and the ``Severo Ochoa Programme for Centers of Excellence in R\&D'' (CEX2019-000904-S). M.F. Espinal is supported by the Agencia Nacional de Investigación y Desarrollo Gobierno de Chile, fellowship number 21190289, and  by Centro de Modelamiento Matemático (CMM), BASAL fund FB210005 for center of excellence from ANID-Chile.\\

\textbf{Data availability statement.} No data were generated or analyzed in the presented research.

\textbf{Conflict-of-Interest Statement.}
 The authors declare no conflict of interest.

\end{document}